\documentclass[10pt]{amsart}

\usepackage{latexsym,exscale,enumerate,amsfonts,amssymb,xparse, mathtools}
\usepackage[normalem]{ulem}
\usepackage{amsmath,amsthm,amsfonts,amssymb,amscd, stmaryrd,textcomp}
\usepackage[bookmarks=false]{hyperref}
\usepackage[usenames,dvipsnames]{xcolor}
\usepackage{enumitem}
\usepackage{verbatim}
\usepackage{ytableau}
\usepackage{euscript}
\usepackage{leftindex}
\usepackage{blkarray}

\colorlet{green}{black!30!green}

\definecolor{myblue}{RGB}{109, 156, 179}

\usepackage{xcolor}
\definecolor{myorange}{rgb}{1,0.647,0}

\definecolor{mypurple}{cmyk}{.6,.9,0, .11}

\numberwithin{equation}{section}

\theoremstyle{definition}
\newtheorem{theorem}{Theorem}[section]
\newtheorem{corollary}[theorem]{Corollary}
\newtheorem{conjecture}[theorem]{Conjecture}
\newtheorem{lemma}[theorem]{Lemma}
\newtheorem{remark}[theorem]{Remark}

\newtheorem{proposition}[theorem]{Proposition}
\newtheorem{definition}[theorem]{Definition}
\newtheorem{example}[theorem]{Example}
\newtheorem{question}[theorem]{Question}

\newtheorem{notation}[theorem]{Notation}

\usepackage{tikz}
\usetikzlibrary{cd}
\usetikzlibrary{snakes}
\usetikzlibrary{decorations.markings}
\usetikzlibrary{decorations.pathreplacing}
\usetikzlibrary{arrows,shapes,positioning}
\usetikzlibrary{decorations.pathmorphing}
\tikzset{anchorbase/.style={baseline={([yshift=-0.5ex]current bounding box.center)}},
tinynodes/.style={font=\tiny,text height=0.75ex,text depth=0.15ex},
smallnodes/.style={font=\scriptsize,text height=0.75ex,text depth=0.15ex},
>={Latex[length=1mm, width=1.5mm]},
overcross/.style={line width=4pt,color=white},
}
\tikzstyle directed=[postaction={decorate,decoration={markings,
	mark=at position #1 with {\arrow{>}}}}]
\tikzstyle rdirected=[postaction={decorate,decoration={markings,
	mark=at position #1 with {\arrow{<}}}}]

\newcommand{\Hom}{{\rm Hom}}

\newcommand{\End}{{\rm End}}

\renewcommand{\to}{\rightarrow}

\newcommand{\id}{{\rm id}}

\let\hat=\widehat
\let\tilde=\widetilde

\def\C{{\mathbb C}}

\newcommand{\gln}[1][n]{\mathfrak{gl}_{#1}}

\newcommand{\slm}{\mathfrak{sl}_{m}}
\newcommand{\glm}{\mathfrak{gl}_{m}}

\newcommand{\son}[1][2n]{\mathfrak{so}_{#1}}
\newcommand{\som}[1][m]{\mathfrak{so}_{#1}}

\newcommand{\ot}{\otimes}

\newcommand{\Br}{\mathrm{Br}}

\newcommand{\flip}{\operatorname{flip}}

\newcommand{\wt}{\mathrm{wt}}
\newcommand{\hgt}{\mathrm{ht}}

\newcommand{\iQW}{\prescript{\iota}{}{T}}

\newcommand{\imag}{\textbf{i}}

\DeclareMathOperator{\BrW}{{\rm Br}_{W}}

\newcommand{\ep}{\varepsilon}

\newcommand{\qbinom}[2]{\begin{bmatrix} #1 \\ #2 \end{bmatrix}}

\newcommand{\g}{\mathfrak{g}}
\newcommand{\borel}{\mathfrak{b}}
\newcommand{\h}{\mathfrak{h}}
\newcommand{\kompact}{\mathfrak{k}}
\newcommand{\hwp}{\boldsymbol{\psi}}
\newcommand{\howespin}{\boldsymbol{\mathbb{S}}}

\begin{document}
%

\title[]{A Kohno--Drinfeld Theorem for $\iota$quantum Weyl groups}

\author{Elijah Bodish}
\address{Department of Mathematics, Indiana University Bloomington, Rawles Hall, Bloomington, IN 47405-47405-7000, USA}
\email{ebodish@iu.edu}

\author{Ivan Motorin}
\address{Department of Mathematics, Massachusetts Institute of Technology, Building 2, Cambridge, MA 02139-4307, USA}
\email{ivanm597@mit.edu}

\begin{abstract}
We prove that two finite-dimensional linear representations of the braid group are isomorphic. One representation comes from the monodromy of the boundary Casimir connection, and the other comes from the $\iota$quantum Weyl group. Both representations are defined for any split symmetric pair $\mathfrak{k}\subset \g$ and for any integrable representation of $\mathfrak{k}$. Our proof uses $(O_m,\son[2n])$ spin Howe duality, so it is only for the pair $\mathfrak{so}_m\subset \mathfrak{sl}_m$.
\end{abstract}

\maketitle

\setcounter{tocdepth}{3}

%
\section{Introduction}
%

Flat connections over a smooth irreducible variety are of considerable interest in modern geometry. When the variety is the complement to a hyperplane arrangement, the most well-studied examples of flat connections are Knizhnik--Zamolodchikov (KZ) \cite{MR853258} and Casimir \cite{MR1797943,MR1926499} connections. Both examples are defined in terms of Lie theory, and Lie theory is used to answer geometric questions about these connections, e.g., in the case of the KZ connection for semisimple $\g$, Schechtman--Varchenko's formulas for flat sections \cite{MR1123378} and Kohno--Drinfeld's description of monodromy in terms of $U_h(\g)$'s $R$-matrix \cite{Drinfeld2,Koh1}. 

In \cite{toledanolaredo2016quasicoxeterquasitriangularquasibialgebrascasimir}, Toledano Laredo describes the monodromy of the Casimir connection $\nabla_{\g}^{Cas}$ in terms of the quantum Weyl group action on a corresponding $U_h(\mathfrak{g})$-module. This work was then extended to the symmetrizable Kac-Moody case by Appel--Toledano Laredo \cite{MR4728239}. An important precedent for those works is Toledano Laredo's \cite[Theorem 7.1]{Valerio-qWeyl}, which proves the result when $\g=\slm$ by using $(\glm,\gln)$ symmetric Howe duality and the Kohno--Drinfeld Theorem \cite{Drinfeld2,Koh1}.

\subsection{Results}

Let $\g$ be a semisimple $\C$-Lie algebra and let $\theta$ be the Chevalley involution:
\[
e_i\mapsto -f_i, \quad f_i\mapsto -e_i, \quad \text{and} \quad h_i \mapsto -h_i.
\]
For the symmetric pair $(\g^{\theta}\subset \g)$ there are analogues of the KZ connection and the Casimir connection: 
\begin{itemize}
    \item the \emph{cyclotomic KZ connection} $\nabla^{{\rm bKZ}}_{\mathfrak{g}^{\theta}\subset \mathfrak{g}}$ \cite{MR2383601}, and
    \item the \emph{boundary Casimir connection} $\nabla^{{\rm bCas}}_{\mathfrak{g}^{\theta}\subset \mathfrak{g}}$ \cite{MR4983792}.
\end{itemize}
The base of the cyclotomic KZ connection $\nabla^{{\rm bKZ}}_{\mathfrak{g}^{\theta}\subset \mathfrak{g}}$ is a type $B$ hyperplane arrangement complement, whereas the base of $\nabla_{\g}^{KZ}$ is a type $A$ arrangement complement. However, the Casimir connections $\nabla_{\g}^{Cas}$ and $\nabla_{\g^{\theta}\subset \g}^{bCas}$ have the same type $\g$ hyperplane arrangement complement in $\h$ as their base.

The notion of a symmetric pair can be lifted to the level of quantum groups. More precisely, an $\iota$quantum group $U_h^{\iota}(\mathfrak{g}^{\theta}\subset \mathfrak{g})$ is a coideal subalgebra of $U_h(\mathfrak{g})$ quantizing $U(\mathfrak{g}^{\theta})$ \cite{W22}. Recently Wang--Zhang developed an $\iota$quantum generalization of Lusztig's quantum Weyl group \cite{WW25}.

\begin{conjecture}[{\cite[Conjecture 8.10]{MR4983792}}]\label{conj1}
    The monodromy representation for $\nabla_{\g^{\theta}\subset \g}^{bCas}$ is isomorphic to the representation given by the $\iota$quantum Weyl group for $U_h^{\iota}(\g^{\theta}\subset \g)$. 
\end{conjecture}

The main result of this paper is the following (see Theorem \ref{monodromy} for a precise statement). 

\begin{theorem}\label{thm:main}
    Conjecture \ref{conj1} is true for the split symmetric pair $(\g^{\theta}\subset \g) = (\som\subset \slm)$.
\end{theorem}

We deduce Theorem \ref{thm:main} using classical and quantum $(\son,O_m)$ spin Howe duality \cite{How95,Wenzl-Spin} as well as the Kohno--Drinfeld Theorem. Our argument is adapted from Toledano Laredo's in \cite{Valerio-qWeyl}, compare Figure \ref{fig:main-proof} to the diagram in \cite[\S1]{Valerio-qWeyl}.

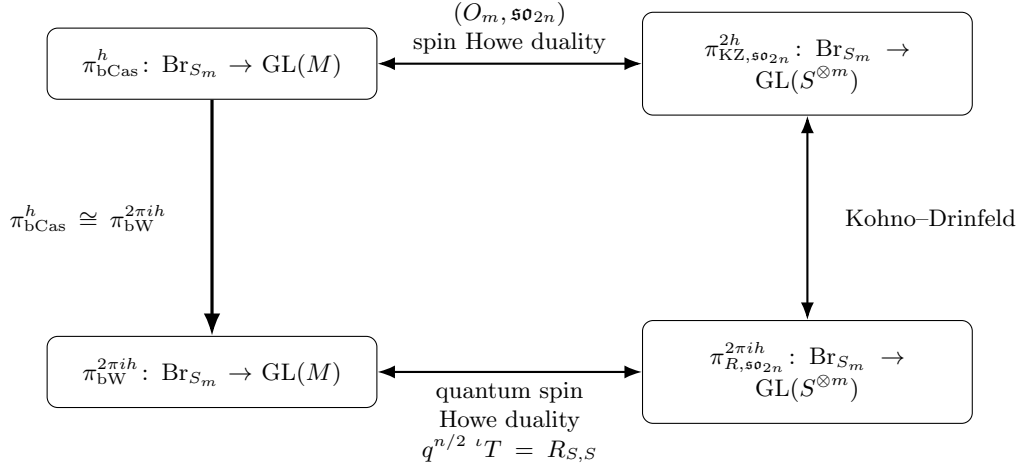
\begin{figure}[ht]
\centering
\begin{tikzpicture}[
    node distance=2.7cm and 3.5cm,
    every node/.style={align=center, font=\small},
    box/.style={
        draw,
        rounded corners,
        inner sep=8pt,
        text width=3.8cm
    },
    iso/.style={<->, >=Latex, thick},
    implication/.style={->, >=Latex, very thick}
]

\node[box] (bcas)
{
\(\displaystyle
\pi_{\mathrm{bCas}}^{h}
\colon
\operatorname{Br}_{S_m}\to\operatorname{GL}(M)
\)
};

\node[box, right=of bcas] (kz)
{
\(\displaystyle
\pi_{\mathrm{KZ},\mathfrak{so}_{2n}}^{2h}
\colon
\operatorname{Br}_{S_m}
\to\operatorname{GL}(S^{\otimes m})
\)
};

\node[box, below=of kz] (rmat)
{
\(\displaystyle
\pi_{R,\mathfrak{so}_{2n}}^{2\pi i h}
\colon
\operatorname{Br}_{S_m}
\to\operatorname{GL}(S^{\otimes m})
\)
};

\node[box, left=of rmat] (iweyl)
{
\(\displaystyle
\pi_{\mathrm{bW}}^{2\pi i h}
\colon
\operatorname{Br}_{S_m}\to\operatorname{GL}(M)
\)
};

\draw[iso]
    (bcas) --
    node[above, text width=3.3cm]
    {\((O_m,\mathfrak{so}_{2n})\)\\spin Howe duality}
    (kz);

\draw[iso]
    (kz) --
    node[right, text width=3cm]
    {Kohno--Drinfeld}
    (rmat);

\draw[iso]
    (rmat) --
    node[below, text width=3.5cm]
    {quantum spin Howe duality\\
     \(q^{n/2}\,\iQW=R_{S,S}\)}
    (iweyl);

\draw[implication]
    (bcas) --
    node[left, text width=3cm]
    {\(\displaystyle
      \pi_{\mathrm{bCas}}^{h}
      \cong
      \pi_{\mathrm{bW}}^{2\pi i h}
     \)}
    (iweyl);

\end{tikzpicture}
\caption{The proof that $\pi_{\mathrm{bCas}}^{h} \cong\pi_{\mathrm{bW}}^{2\pi i h}$ reduces to the Kohno--Drinfeld theorem using classical and quantum spin Howe
duality.}
\label{fig:main-proof}
\end{figure}

\subsection{Further directions}

This work led to many discussions with experts from which we synthesized the following questions concerning the Casimir connection.

\begin{question}
    In \cite{BER, bodish2026typebwebs}, there is an $\iota$quantum Weyl group action on nonclassical\footnote{Following the classical/nonclassical naming convention of Iorgov--Klimyk \cite{MR2143754}.} $U_{h+\pi \imag}^{\iota}(\som\subset \slm)$-modules which identifies with the $\son$ $R$-matrix under quantum\footnote{There is no classical Howe duality in this case. When $h=0$, the commuting algebras are $U(\son[2n+1])$ and $U_{-1}(\som)$.} $(U_{h}(\son[2n+1]), U_{2h+\pi \imag}^{\iota}(\som))$ spin Howe duality. Do the results of \cite{MR4983792} and the arguments of this paper generalize to non-integrable modules?
\end{question}

\begin{question}\label{quest:quasisplit}
    Wang--Zhang define the $\iota$quantum Weyl group for quasi-split symmetric pairs \cite{WW25}, although we only use their definition for split pairs since that is the setting in which the boundary Casimir has been defined \cite{MR4983792}. Is there a boundary Casimir connection for quasi-split (or even more general) symmetric pairs? One might hope to relate this to the duality studied in \cite{MR5020472}.
\end{question}

\begin{question}
Let $\Upsilon_i$ be the rank one $1$-tensor quasi-$K$ matrix \cite{MR3864017,MR3905136}, $\iQW_i$ be the $\iota$quantum Weyl group generator, and $T_i$ be Lusztig's quantum Weyl group generator. Wang--Zhang proved that as operators on integrable $U_{h}^{\iota}$-modules 
\begin{equation}\label{eq:TKT}
\iQW_i = \Upsilon_iT_i,
\end{equation}
see \cite[Theorem B]{WW25}. Is there a relation between $\nabla_{\g^{\theta}\subset \g}^{bCas}$ and $\nabla_{\g}^{Cas}$ realizing \eqref{eq:TKT}?
\end{question}

\begin{question}
    Let $\Theta^{\iota}$ denote the $2$-tensor quasi-$K$ matrix \cite[Equation 3.1]{MR3864017}. Dobson--Kolb's results on factorization of the quasi-$K$ matrix \cite{MR4021849} suggest the following $\iota$quantum analogue of the Levendorskii--Soibelman/Kirillov--Reshetikhin formula \cite{KirResh, MR1142216}:
    \[
\Theta^{\iota}|_{M\otimes V}= (\iQW_{w_0}^{-1}\otimes T_{w_0}^{-1})\circ \Delta(\iQW_{w_0})|_{M\otimes V},
    \]
    where $V$ is a $U_h(\g)$-module and $M$ is an integrable $U_h^{\iota}$-module. Is there a related (useful) notion of a braided module Coxeter category extending Appel--Toledano Laredo's braided Coxeter categories \cite{MR3984102}?
\end{question}

\begin{question}
    It is possible to rewrite KZ equations in the context of certain interpolation categories known as Deligne categories. In particular, by using the orthogonal Howe duality the KZ equations in the context of a Deligne category $\underline{\smash{{\rm Rep}}}(O_t), \ t\in \C \setminus \mathbb{Z}$, correspond to the boundary Casimir connection \cite{MR5067248}. Is there an analogue of these results for the spin Brauer category \cite{MR4824565, mcnamara2025quantumspinbrauercategory}?
\end{question}

\begin{question}
    Is there a trigonometric boundary Casimir connection generalizing \cite{MR3087943,MR2769327}? In this analogy, the Yangian would be replaced with the twisted Yangian, and the monodromy would be given by the $\iota$quantum Weyl group of $U_q^{\iota}\subset U_q(\hat{\g})$.
\end{question}

    Maulik-Okounkov found the trigonometric Casimir connection is related to quantum cohomology of quiver varieties \cite{MR3951025}. Danilenko also found that the trigonometric KZ connection is related to quantum cohomology of affine Grassmannian slices in types ADE \cite{MR4779057}. Since the KZ connection in type $D$ is spin Howe dual to the boundary Casimir connection for $\som\subset \slm$, this is some evidence there is an enumerative geometry realization of the (to-be-defined) trigonometric boundary Casimir.
    
\subsection*{Structure of the paper}

Section \ref{sec2} fixes notations and reviews basic facts about Lie algebras and boundary Casimir connections. Section \ref{sec3} reviews Knizhnik–Zamolodchikov connection, Clifford algebras, and classical orthogonal Howe duality. Section \ref{sec4} gives necessary background on quantum and $\iota$quantum groups. Section \ref{sec5} recalls Wenzl's approach to the quantum analogue of orthogonal Howe duality. Section \ref{sec6} is reserved for computations relating braid group $R$-matrix action and the action of the $\iota$quantum Weyl group. Section \ref{sec7} contains the statement and proof of our main result on monodromy of the boundary Casimir connection for the symmetric pair $(\som\subset\slm)$.

\subsection*{Acknowledgments} 
We thank Artem Kalmykov, whose collaboration with the first named author lead to this project, for their careful reading of an earlier draft. We also thank Pavel Etingof, Valerio Toledano Laredo, Weiqiang Wang, and Weinan Zhang for many useful discussions. E.B. was partially supported by NSF MSPRF-2202897. We used AI to create Figure \ref{fig:main-proof} and for proofreading. 

%
\section{Boundary Casimir connection}\label{sec2}
%
\subsection{Lie algebras and root systems}\label{sec:lietheorynotation}

Fix a semisimple complex Lie algebra $\g$. Choose a Cartan subalgebra $\h\subset \g$, with roots $\Phi \subset \h^*$, and a Borel subalgebra $\h\subset \borel\subset \g$, with positive roots $\Phi_+\subset \Phi$ and simple roots $\Delta\subset \Phi_+\subset \Phi$.

The Killing form $\kappa:\g\times \g \rightarrow \mathbb{C}$ induces an isomorphism $\h\rightarrow \h^*$. Write $t_{\alpha} \in \h$ to denote the element which this isomorphism identifies with $\alpha\in \Phi$, i.e., such that $\kappa(t_{\alpha}, h) = \alpha(h)$ for all $h\in \h$, and write $\kappa^*(\alpha,\beta) :=\kappa(t_{\alpha},t_{\beta})$ for $\alpha, \beta\in \Phi$. For $\alpha,\beta\in \Delta$ we denote the entries of the Cartan matrix by $a_{\alpha\beta}=2\kappa^*(\alpha,\beta)/\kappa^*(\alpha,\alpha)$. The symmetrizing integers $d_{\alpha}\in \{1,2,3\}$ are defined to be proportional to $\kappa^*(\alpha,\alpha)$, such that $d_{\alpha}=1$ for all short roots. For an enumeration of simple roots $\{\alpha_1, \dots, \alpha_{|\Delta|}\}=\Delta$, the fundamental weights are the elements $\varpi_i\in X$, such that $2\kappa^*(\alpha_i,\varpi_j)/\kappa^*(\alpha_i,\alpha_i)=\delta_{ij}$.

We use the notation $\Lambda=\mathbb{Z}\Phi$ for the root lattice. Write $\Phi^{\vee}\subset \h$ for the dual root system with coroots $\alpha^{\vee}=2t_{\alpha}/ \kappa^*(\alpha,\alpha)$, for all $\alpha\in \Phi$. Let $Y=\mathbb{Z}\Phi^{\vee}$ be the coroot lattice, and let $X\subset \mathbb{Q}\Phi$ be the weight lattice, such that there is a perfect bilinear pairing $\langle -,- \rangle: Y \times X \rightarrow \mathbb{Z}$. 

 The choice of simple roots $\Delta$ yields the set of simple coroots $\Delta^{\vee}$ and the set of positive coroots $\Phi^{\vee}_+$. We also denote by $\rho$ and $\rho^{\vee}$ the half-sum of positive roots $\Phi_+$ and positive coroots $\Phi^{\vee}_+$, respectively. Finally, $W$ denotes the Weyl group acting on both $\Phi$ and $\Phi^{\vee}$.

For $x\in \g_{\alpha}$ and $y\in \g_{-\alpha}$, we have $[x,y] = \kappa(x,y)t_{\alpha}$. Set $h_{\alpha}:=\frac{2t_{\alpha}}{\kappa(t_{\alpha}, t_{\alpha})}$. For any $x_{\alpha}\in \g_{\alpha}$, there is $x_{-\alpha}\in \g_{-\alpha}$ such that $x_{\alpha}$, $x_{-\alpha}$, and $h_{\alpha}$ are an $\mathfrak{sl}_2$-triple \cite[Section 8.3]{MR323842}. Note that $\frac{2t_{\alpha}}{\kappa(t_{\alpha}, t_{\alpha})}=h_{\alpha}=[x_{\alpha},x_{-\alpha}]=\kappa(x_{\alpha},x_{-\alpha})t_{\alpha}$, so $\kappa(x_{\alpha},x_{-\alpha}) = \frac{2}{\kappa^*(\alpha,\alpha)}$. 
      
Fix a Chevalley basis \cite[Section 25.2]{MR323842} $\{x_{\alpha} \ | x_{\alpha}\in \g_{\alpha}, \alpha \in \Phi\}\cup \{h_{\alpha_i} | h_{\alpha_i}\in \h, \alpha_i\in \Delta\}$ such that 
\begin{enumerate}
    \item For $\alpha \in \Phi$, $[x_{\alpha},x_{-\alpha}] = h_{\alpha}\in \mathbb{Z}\cdot\{h_{\alpha_i} \ | \ \alpha_i\in \Delta\}$.
    \item For $\alpha, \beta, \alpha+\beta\in \Phi$, $[x_{\alpha},x_{\beta}] = c_{\alpha,\beta}x_{\alpha+\beta}$, and $[x_{-\alpha},x_{-\beta}] = -c_{\alpha,\beta}x_{-(\alpha+\beta)}$.
    \item For $\alpha, \beta, \alpha+\beta\in \Phi$, $c_{\alpha, \beta} = \pm (r+1)\ne 0$, where $\beta - r\alpha, \dots, \beta+q\alpha$ is the $\alpha$-string through $\beta$.
\end{enumerate} 
There is also a normalized basis 
\[
E_{\alpha}:=\sqrt{\kappa^*(\alpha,\alpha)/2}\cdot x_{\alpha}, \quad  F_{\alpha}:=\sqrt{\kappa^*(\alpha,\alpha)/2}\cdot x_{-\alpha}, \quad t_{\alpha_i}, \quad \alpha\in \Phi_+,\quad \alpha_i\in \Delta,
\]
such that $[E_{\alpha_i},F_{\alpha_i}]=t_{\alpha_i}$ and $\kappa(E_{\alpha},F_{\alpha})=1$.

\subsection{Symmetric pairs and boundary Casimir connections}

For $\alpha_i\in \Delta$, write $e_i:=x_{\alpha_i}$, $f_i:=x_{-\alpha_i}$, and $h_{i}:=h_{\alpha_i}$. Then $\g$ has a Serre presentation, with generators $e_{i}$, $f_{i}$, and $h_{i}$. The relations in the presentation make clear there is a Chevalley automorphism
\begin{equation}\label{eq:Chevalleyauto}
\theta: \g\rightarrow \g \quad \text{such that} \quad e_{i}\mapsto -f_{i}, \quad f_{i}\mapsto -e_{i}, \quad \text{and}  \quad h_{i} \mapsto - h_{i}, \quad \text{for $\alpha_i\in \Delta$.}
\end{equation}

\begin{lemma}
    If $\alpha \in \Phi$, then $\theta(x_{\alpha})= -x_{-\alpha}$.
\end{lemma}
\begin{proof}
    If $\alpha\in \Delta$, this is clear. If the height of a root is bigger than one (so the root is not simple), then that root is of the form $\alpha+\beta$, where $\alpha, \beta$ have smaller height. For $\alpha, \beta, \alpha+\beta\in \Phi$, the Chevalley basis structure constants satisfy $c_{\alpha,\beta}\ne 0$ and $c_{-\alpha, -\beta} = - c_{\alpha, \beta}$, so if $\theta(x_{\alpha})=-x_{-\alpha}$ and $\theta(x_{\beta})= -x_{-\beta}$, then
    \[
    \theta(x_{\alpha+\beta})= c_{\alpha,\beta}^{-1}\theta([x_{\alpha},x_{\beta}]) = c_{\alpha,\beta}^{-1}[-x_{-\alpha},-x_{-\beta}] = -c_{-\alpha,-\beta}^{-1}[x_{-\alpha},x_{-\beta}] = -x_{-(\alpha+\beta)}.
    \]
    Using the above, the claim follows from induction on height. 
\end{proof}

\begin{definition}
    For $\alpha\in \Phi$, define $b_{\alpha}:= x_{-\alpha} - x_{\alpha}$. Note that $b_{-\alpha} = -b_{\alpha}$. We also consider normalized versions $B_{\alpha}:=F_{\alpha}-E_{\alpha}$, for $\alpha\in \Phi_+$, and $B_{-\alpha}:=-B_{\alpha}$. For $\alpha_i\in \Delta$, we will write $b_{i}:=f_i-e_i$.
\end{definition}

\begin{remark}
    For all $\alpha\in \Phi$, $B_{\alpha}^2=\frac{\kappa^*(\alpha, \alpha)}{2}b_{\alpha}^2$.
\end{remark}

Let $\kompact:=\g^{\theta}$ be the linear subspace of $\g$ fixed by $\theta$, so $\kompact = \C\cdot\{b_{\alpha} \ | \ \alpha \in \Phi_+\}$. Since $\theta$ is a Lie algebra automorphism, the subspace $\kompact$ is a Lie subalgebra. Moreover, $\kompact$ is generated as a Lie algebra by $\{b_{i} \ | \ \alpha_i\in \Delta\}$. The $\theta$-anti-invariant subspace is the $\C$-span of $\{x_{-\alpha}+x_{\alpha} \ | \ \alpha \in \Phi_+\}\cup\{h_i \ | \ \alpha_i \in \Delta\}$.

\begin{remark}
    Such pairs $(\g, \kompact)$ are called split symmetric pairs; they satisfy the decomposition $\g= \kompact \oplus \borel$.
\end{remark} 

\begin{definition}
    Let $h\in \mathbb{C}$. The \emph{boundary Casimir one-form} $\nabla_{bCas, \kompact\subset \g}^{h}$ is
    \begin{equation}
    \nabla_{bCas, \kompact\subset \g}^{h}:= d- h\sum_{\alpha\in \Phi_+}\frac{d\alpha}{\alpha}B_{\alpha}^2. 
    \end{equation}
\end{definition}

\begin{remark}
    We can also write $\nabla_{bCas, \kompact\subset \g}^{h}=d- \frac{h}{2}\sum_{\alpha\in \Phi}\frac{d\alpha}{\alpha}B_{\alpha}^2$.
\end{remark}
    
    Recall that
    \[
    \h^{reg} = \h\setminus \bigcup_{\alpha \in \Phi_+}\{h\in \h \ | \ \alpha(h) = 0\}.
    \] 
    Let $M$ be a representation of $\kompact$. The one-form $\nabla_{bCas, \kompact\subset \g}^{h}$ defines a regular connection on the vector bundle $\h^{reg}\times M\rightarrow \h^{reg}$, which we denote by $\nabla_{bCas, \kompact\subset \g}^{h, M}$ and refer to as the \emph{boundary Casimir connection}.

\begin{theorem}
    The connection $\nabla_{bCas,\kompact\subset \g}^{h, M}$ is flat.
\end{theorem}
\begin{proof}
Using Kohno's flatness criterion for connections on hyperplane complements \cite{MR1046571}, as stated in \cite[Lemma 2.1]{Valerio-qWeyl}, this is shown in \cite[Theorem 8.7]{MR4983792}.
\end{proof}

To study the monodromy of $\nabla_{bCas,\kompact\subset \g}^{h,M}$, we consider fundamental groups and braid groups related to $\g$.

Let $G$ be the complex, connected, and simply-connected Lie group, with Lie algebra $\g$. Let $T\subset G$ be the maximal torus with $\mathrm{Lie}(T) = \h$. Denote by $K$ the complex Lie subgroup of $G$ generated by $\exp(t b_{\alpha})$, for $\alpha\in \Delta$ and $t\in \C$. Note that $\mathrm{Lie}(K) = \kompact$.

For any semisimple Lie algebra $\mathfrak{g}$, we will be working with the category ${\rm Rep}(\mathfrak{g})$ of finite-dimensional integrable representations of $\mathfrak{g}$. We call a representation of $\kompact$ integrable if it is the differential of a representation of $K$ (see Definition \ref{def:iquantumintegrable} for a more general notion of integrable $\kompact$-module).

The Weyl group $W$ is a Coxeter group, so it has an associated braid group $\BrW$, with simple braid generators $\sigma_i$, for $\alpha_i \in \Delta$. Let $\tilde{s}_{i}:=\exp(-e_i)\exp(f_i)\exp(-e_i)\in N(T)$ be the standard lift of simple reflection generators of the Weyl group $W = N(T)/T$. The map $\BrW\longrightarrow N(T)$ sending $\sigma_{i}\mapsto \tilde{s}_{i}$ determines a group homomorphism, since the elements $\tilde{s}_{i}$ satisfy the braid relations in $N(T)$.

\begin{remark}\label{rem:Wactiononh-permutescoords}
    The $W$ action on $\h$ is realized by letting $s_i$ act on $\h$ as conjugation by $\tilde{s}_i$. In particular, for $\h\subset \glm$, the action of $W$ permutes coordinates, e.g.,
    $\begin{pmatrix}
    0 & -1\\
    1 & 0
\end{pmatrix}\cdot
\begin{pmatrix}
    a & 0\\
    0 & b
\end{pmatrix}\cdot
\begin{pmatrix}
    0 & -1\\
    1 & 0
\end{pmatrix}^{-1}=
\begin{pmatrix}
    b & 0\\
    0 & a
\end{pmatrix}$ 
\end{remark}

The group $\BrW$ also appears as a fundamental group \cite{MR293615}. Let $p:\tilde{\h^{reg}}\rightarrow \h^{reg}$ denote the universal cover. The action of $W$ on $\h^{reg}$ is fixed point free. Composition with $\pi_W:\h^{reg}\rightarrow \h^{reg}/W$ yields the universal cover of $\h^{reg}/W$, $\pi_W\circ p:\tilde{\h^{reg}}\rightarrow \h^{reg}/W$. Then $\BrW\cong \Br_{\g}:=\pi_1(\h^{reg}/W)$ acts on $\tilde{\h^{reg}}$.

\begin{remark}
The fundamental group of $\tilde{\h^{reg}}$ is the pure braid group, which is the kernel of the group homomorphism $\BrW\rightarrow W$.
\end{remark}

There is a nontrivial intersection between $N(T)$ and $K$.

\begin{lemma}\label{lift_of_braid_gp_gens}
    If $\alpha_i\in \Delta$, then $\tilde{s}_{i} = \exp(\frac{\pi}{2}b_{i})\in K$.
\end{lemma}
\begin{proof}
This follows from \cite[Proposition 3.5]{MR4983792}.
\end{proof}

Thus, there is a group homomorphism $\sigma:\BrW\longrightarrow K$, such that $\sigma_{i}\mapsto \exp(\frac{\pi}{2}b_{i})$. In particular, if $M$ is an integrable representation of $K$, then $\BrW$ acts on $M$. In particular, $\BrW$ acts on $\kompact$.

\begin{lemma}\label{lem:wBwinvers}
    If $\gamma\in \BrW$ projects to $w\in W$, then $\sigma(\gamma)B^2_{\alpha}\sigma(\gamma)^{-1} =  B^2_{w(\alpha)}$, for $\alpha \in \Phi$. 
\end{lemma}
\begin{proof}
    It suffices to show $\tilde{s_{i}}B_{\alpha}\tilde{s_{i}}^{-1} = \pm B_{s_{i}(\alpha)}$, for $\alpha_i \in \Delta$ and $\alpha \in \Phi$, which follows from \cite[Lemma 3.15]{MR4983792}.
\end{proof}

\begin{proposition}\label{prop:iCasconnectiondescends}
    Let $M$ be an integrable $\kompact$-module. The one-form $p^*\nabla_{bCas, \kompact\subset \g}^{h}$ defines a $\BrW$-equivariant flat connection on $\tilde{\h^{reg}}\times M=p^*(\h^{reg}\times M)$, which therefore descends to a flat connection on the vector bundle $\tilde{\h^{reg}}\times_{\BrW} M\rightarrow \h^{reg}/W$.
\end{proposition}

\begin{proof}
    Using Lemma \ref{lem:wBwinvers}, this argument is identical to the proof of \cite[Proposition 2.3]{Valerio-qWeyl}.
\end{proof}

We write $\pi^{h}_{bCas}:\BrW \rightarrow GL(M)$ to denote the monodromy representation of the boundary Casimir connection from Proposition \ref{prop:iCasconnectiondescends}. (Note that we have fixed $\sigma:\BrW \rightarrow K$.)

\begin{proposition}\label{parralel_transport}
    Let $M$ be an integrable $\kompact$-module. Let $\gamma\in \BrW=\pi_1(\mathfrak{h}^{reg}/W)$ and $\widetilde{\gamma}:[0,1]\rightarrow \mathfrak{h}^{reg}$ be a lift of $\gamma$, then
    \begin{equation}\label{monodromy_bCas_sigma}
        \pi^{h}_{bCas}(\gamma)=\sigma(\gamma)\mathcal{P}(\widetilde{\gamma})
    \end{equation}
    where $\mathcal{P}(\widetilde{\gamma})\in GL(M)$ is the parallel transport along $\widetilde{\gamma}$ for the connection $\nabla_{bCas,\kompact \subset \mathfrak{g}}^{h,M}$ on $\mathfrak{h}^{reg} \times V$ and $\sigma$ is the representation of $\Br_W$ on $M$ which factors through $\sigma:\Br_W\rightarrow K$.
\end{proposition}
\begin{proof}
    Same as in the proof of \cite[Proposition 2.4]{Valerio-qWeyl}.
\end{proof}
%
\section{Knizhnik–Zamolodchikov for $\son$ and the dual pair $(O_m, \son)$}\label{sec3}
%

\subsection{Knizhnik–Zamolodchikov connection}
Fix a simple complex Lie algebra $\g$ and a non-degenerate invariant form, i.e., a non-zero multiple of the Killing form. Let $\{e_a\}$ and $\{e^a\}$ be dual bases in $\g$ with respect to the chosen form. The Casimir element in the center of $U(\g)$ is $C = \sum_{1\le a \le \dim \g} e_ae^a$, while the split Casimir in $ U(\g)\otimes U(\g)$ is the tensor 
\[
\Omega:= \sum_{1\le a\le \dim \g}e_a\otimes e^a. 
\]
It is a standard exercise to check that $C$ and $\Omega$ do not depend on the choice of basis.

Fix an integer $r\ge 2$. For $1\le i\ne j \le r$, we define a tensor $\Omega_{ij}\in U(\g)^{\otimes r}$ as 
\[
    \Omega_{ij}:=\sum_{1\le a\le \dim(\g)}1\otimes \dots \otimes 1\otimes e_a \otimes 1 \otimes \dots \otimes 1\otimes e^{a} \otimes 1 \otimes \dots \otimes 1,
\]
where the $e_a$ and $e^a$ terms appear in the $i$-th and $j$-th tensor factors respectively. Again, $\Omega_{ij}$ does not depend on the choice of basis, so in particular, $\Omega_{ij}= \Omega_{ji}$.  

\begin{definition}
Let $h\in \C$. The \emph{KZ connection one-form} $\nabla_{KZ, \g}^{h}$ is defined as
\begin{equation}\label{KZ_defn}
    \nabla_{KZ, \g}^{h}:=d-h \sum_{1\le i<j \le r} \frac{d(z_i-z_j)}{(z_i-z_j)}\Omega_{ij}. 
\end{equation}
\end{definition}

\begin{remark}
    We can also write $\nabla_{KZ, \g}^{h} = d-\frac{h}{2}\sum_{1\le i\ne j\le r}\frac{d(z_i-z_j)}{(z_i-z_j)}\Omega_{ij}$.
\end{remark}

Recall that
\[
Conf_r(\C):=
    \mathbb{C}^r \setminus \bigcup_{1\le i\ne j\le r} \{z \in \mathbb{C}^r \ | \ z_i-z_j=0\}
\quad \text{and} \quad
UConf_r(\C):=Conf_r(\C)/ S_r,
\]
where $S_r$ acts on $Conf_r(\C)$ by permuting coordinates. Upon choosing a base point in $UConf_r(\C)$ we have an identification of $\pi_1(UConf_r(\C))$ with the braid group ${\rm Br}_{ S_r}$. 

Fix $V_1, \dots, V_r \in {\rm Rep}(\mathfrak{g})$. The one-form $\nabla_{KZ, \g}^{h}$ defines a flat connection $\nabla_{KZ, \g}^{h, V_1, \dots, V_r}$ on the trivial vector bundle over $Conf_r(\C)$ with fibers $V_1\otimes \dots \otimes V_r$, which we refer to as the \emph{KZ connection}. When $V_1=\dots = V_r= V$, the symmetric group $S_r$ acts on $V^{\otimes r}$ by flipping tensor factors, and $\nabla_{KZ, \g}^{h, V, \dots, V}$ descends to a flat connection on the vector bundle $Conf_r(\C)\times_{{S_r}}V^{\otimes r}\rightarrow UConf_r(\C)$.

Recall that $UConf_r(\C)$ is homotopy equivalent to $\mathfrak{h}^{reg}/W$ for $\mathfrak{sl}_r$. So the fundamental groups of both spaces are isomorphic (see the discussion after Proposition \ref{Gauge_equiv}). We have a similar Proposition to \ref{parralel_transport}.

\begin{proposition}\label{parralel_transport_KZ}
    Let $\gamma \in \pi_1(UConf_r(\C))$ and $\widetilde{\gamma}:[0,1]\rightarrow Conf_r(\C)$ be a lift of $\gamma$, then the monodromy representation of the KZ connection $\pi_{KZ}^h:\Br_{ S_r} \rightarrow GL(V^{\ot r})$ satisfies the relation
    \[
    \pi_{KZ}^h(\gamma)=\sigma'(\gamma)\mathcal{P}(\widetilde{\gamma})
    \]
    where $\mathcal{P}(\widetilde{\gamma})\in GL(V^{\ot r})$ is the parallel transport along $\widetilde{\gamma}$ for it and $\sigma'$ is the representation of $\Br_{ S_r}$ on $V^{\ot r}$ which factors through the surjection $\Br_{ S_r}\rightarrow  S_r$.
\end{proposition}

\subsection{Clifford algebra $Cl(2n)$ and spin representation of $\son$}\label{spinrepforson}

Let $(\C^{2n}, (-,-))$ be the vector space, with basis $e_1, \dots, e_{2n}$, equipped with the nondegenerate symmetric bilinear form $(e_i, e_j) := \delta_{i,j}$. 

\begin{definition}
The Lie algebra $\son[2n]$ is the $\C$-vector space of all $A\in \End(\C^{2n})$ preserving the form $(-,-)$, i.e., such that $(A(-),-) + (-,A(-))=0$.
\end{definition}

\begin{definition}
The \emph{Clifford algebra} $Cl(2n)$ is the unital associative $\C$-algebra generated by $\C^{2n}$ modulo the relation $v^2= (v,v)$ for all $v\in \C^{2n}$.
\end{definition} 

A consequence of the defining relation is that $vw+wv = 2(v,w)$, for $v,w\in \C^{2n}$. In particular, we have $e_i^2=1$, for $1\le i\le 2n$, while $e_ie_j=-e_je_i$, for $1\le i\ne j\le 2n$. 

We now recall how to relate $\son[2n]$ to $Cl(2n)$. Note that for $a,b\in \C^{2n}$, the operator $\varphi_{a\wedge b}:\C^{2n}\rightarrow \C^{2n}$, defined by 
\begin{equation}\label{eq:varphisubab}
\varphi_{a\wedge b}(v) = 2\Big((b,v)a - (a,v)b\Big),
\end{equation}
preserves the form $(-,-)$. Thus, the assignment $a\wedge b\mapsto \varphi_{a \wedge b}$ gives an isomorphism $\Lambda^2(\C^{2n})\cong \son[2n]$.

\begin{lemma}
    The map 
    \[
    \son[2n]\rightarrow Cl(2n), \quad \varphi_{a\wedge b}\mapsto \frac{ab-ba}{2}
    \]
    is a Lie algebra homomorphism.
\end{lemma}
\begin{proof}
We have the following explicit description of a Lie bracket on $\Lambda^2(\C^{2n})\cong \son$, corresponding to the commutator on $\mathfrak{so}_{2n}$: $[a\wedge b, c\wedge d] := 2(b,c)a\wedge d -2(a,c)b\wedge d - 2(b,d)a\wedge c +2(a,d)b\wedge c$. Computing in $Cl(2n)$ we find that: $[ab,cd] = abcd-cdab = 2(b,c)ad -2(a,c)bd - 2(b,d)ac +2(a,d)bc$. Recall the Clifford quantization map $\Lambda^*(\C^{2n})\rightarrow Cl(2n)$, $a_1\wedge \dots \wedge a_k\mapsto \frac{1}{k!}\sum_{w\in S_k}(-1)^wa_{w(1)}\dots a_{w(k)}$, which is an isomorphism of vector spaces. The restriction of the Clifford quantization map $\Lambda^2(\C^{2n})\rightarrow Cl(2n)$, $
a\wedge b\mapsto (ab-ba)/2=ab-(a,b)\cdot 1$, 
is therefore a Lie algebra homomorphism.
\end{proof}

If we instead work with the following basis of $\C^{2n}$ (here ${\bf i}:=\sqrt{-1}$):
\[
\psi_i:= \frac{e_{2i-1}+{\bf i}e_{2i}}{2} \qquad \text{and} \qquad \psi_i^*:= \frac{e_{2i-1}-{\bf i}e_{2i}}{2}, \qquad 1\le i\le n,
\]
which satisfies
\begin{equation}\label{eq:pairingpsi}
(\psi_i,\psi_j)= 0 \qquad \text{and} \qquad (\psi_i,\psi_j^*) = \frac{1}{2}\delta_{ij} \qquad 1\le i,j\le n, 
\end{equation}
then we have the following. 

\begin{lemma}
    Using the basis $\psi_1, \dots, \psi_n, \psi_n^*, \dots, \psi_1^*$ to identify $\End(\C^{2n})\cong M_{2n}(\C)$, we have
    \[
    \son = \{A\in M_{2n}(\C) \ | \ Aw_0+w_0A^T=0\} = \C\cdot\{M_{i,j}:=E_{i,j}-E_{-j,-i} \ | \ i,j\in \{\pm 1, \dots, \pm n\}\},
    \]
    where $w_0$ is the anti-diagonal matrix, with all anti-diagonal entries $1$, and $E_{i,j}$ are matrix units.
\end{lemma} 

\begin{lemma}\label{lem:cliffordgensrelns}
The Clifford algebra $Cl(2n)$ is generated by $\psi_i, \psi_i^{*}, \quad 1\le i\le  n$
with relations
\begin{equation}\label{eq:Cliffrelns}    \psi_i\psi^{*}_j+\psi^{*}_j\psi_i=\delta_{ij}1, \quad \psi_i\psi_j+\psi_j\psi_i=0, \quad \psi_i^{*}\psi_j^{*}+\psi_j^{*}\psi_i^{*}=0, \quad  1\le i,j \le n.
\end{equation}
\end{lemma}

Write $\psi_{-i}:=\psi_i^*$, for $-i< 0$. Identifying $\End(\C^{2n})\cong M_{2n}(\C)$, via the basis $\psi_1, \dots, \psi_n, \psi_{-n}, \dots, \psi_{-1}$, the isomorphism $\Lambda^2(\C^{2n})\cong \son$ is such that $\psi_i\wedge\psi_{-j}\mapsto \varphi_{\psi_i\wedge\psi_{-j}} = M_{i,j}$ (the factor of $2$ in \eqref{eq:varphisubab} cancels the $\frac{1}{2}$ in \eqref{eq:pairingpsi}). Thus, we have the following.

\begin{lemma}\label{lem:so2ntocliff}
The assignment
\[
M_{i,j}\mapsto \frac{\psi_{i}\psi_{-j}-\psi_{-j}\psi_{i}}{2} \stackrel{\eqref{eq:Cliffrelns}}{=} \psi_{i}\psi_{-j} - \frac{1}{2}\delta_{i,j}, \qquad i,j\in \{\pm 1, \dots, \pm n\},
\]
defines a map $\son[2n]\rightarrow Cl(2n)$, inducing an algebra homomorphism $U(\son[2n])\rightarrow Cl(2n)$. 
\end{lemma}

The presentation in Lemma \ref{lem:cliffordgensrelns} makes apparent that the algebra $Cl(2n)$ has a spin representation.

\begin{definition}\label{def:spinmodule}
    Let $I\subset Cl(2n)$ be the left ideal generated by $\{\psi_1, \dots, \psi_n\}$. The Clifford algebra's \emph{spin module} is $S:=Cl(2n)/I$. Since $\son$ maps to $Cl(2n)$, the Lie algebra $\son$ has a \emph{spin representation} $S$.
\end{definition} 

\begin{notation}\label{not:orderedbasisforS}
    For any $I\subset \{1,\dots,n\}$, associate the following vector in $S$
     \begin{equation*}
        \hwp_{I}:= \psi_{i_1}^{*} \dots \psi_{i_{|I|}}^{*}, \quad I=\{i_1<i_2<\dots < i_{|I|}\}.
    \end{equation*}
    The set $\{\hwp_I\}_{I\subset \{1, \dots, n\}}$ is a basis of $S$. 
\end{notation}

Note that $M_{i,j} = -M_{-j,-i}$, so $\{M_{i,j}\}$ is only a spanning set. Now, we will refine this spanning set to a basis. We choose a Cartan subalgebra $\h_{\son[2n]}$ of diagonal matrices in $\son[2n]$
\begin{equation}\label{eq:basisso2ncartan}
h=h_{1}M_{1,1} + h_2M_{2,2} + \dots + h_nM_{n,n}.
\end{equation}
Denote by $\epsilon_i \in \h_{\son[2n]}^*$ the functional $\epsilon_i(h) = h_i$. Then the roots spaces of $\g = \son[2n]$ are
\begin{equation}\label{eq:basisso2nroots}
		\g_{\epsilon_i - \epsilon_j} = \C\cdot M_{ij}, \qquad \g_{\epsilon_i + \epsilon_j} = \C\cdot M_{i,-j}, \qquad\text{and} \qquad \g_{-\epsilon_i-\epsilon_j} = \C\cdot M_{-j,i},
\end{equation}
		for $i\neq j\in\{1, \dots, n\}$. 

\begin{remark}\label{rem:dualmij}
    Since $\son$ is a simple Lie algebra it has a unique nondegenerate invariant form (up to scalar). It is convenient to use the trace form in the defining representation, which is computed by the formula $(E_{i,j},E_{k,\ell})=\delta_{j,k}\delta_{i,\ell}$. In this case, the dual basis is $M^{i,j}:=\frac{1}{2}M_{j,i}$.
\end{remark}
        
        Our choice of positive roots will be
		\[
		\Phi_+ = \{ \epsilon_i \pm \epsilon_j | 1\le i< j\le n\}
		\]
		with simple roots
		\[
		\alpha_1 = \epsilon_1 - \epsilon_2,\quad \alpha_2 = \epsilon_2 - \epsilon_3,\quad \dots, \quad  \alpha_{n-1} = \epsilon_{n-1} - \epsilon_n, \quad \text{and} \quad \alpha_n = \epsilon_{n-1} + \epsilon_n.
		\]
It is useful to keep the following expression for fundamental weights in terms of $\epsilon_i$:
\[
\varpi_{i}=\epsilon_1+\dots + \epsilon_{i},\quad  1\le i\le n-2, \quad \varpi_{n-1}=\frac{\epsilon_1}{2}+\dots+\frac{\epsilon_{n-1}}{2}-\frac{\epsilon_{n}}{2}, \quad \varpi_{n}=\frac{\epsilon_1}{2}+\dots+\frac{\epsilon_{n-1}}{2}+\frac{\epsilon_{n}}{2}.
\]
For dominant integral $\lambda$, we have $V_{\lambda}$, the finite-dimensional irreducible of highest weight $\lambda$.

\begin{remark}
The basis $\{\hwp_I\}_{I\subset\{1, \dots, n\}}$ of $S$ is a weight basis with respect to $\h_{\son[2n]}$, and
\[
\wt(\hwp_{\{i_1, \dots , i_k\}}) = \wt(\psi^*_{i_1}\cdots \psi^*_{i_k}) = \big(\frac{\epsilon_1}{2}+\dots+\frac{\epsilon_{n-1}}{2}+\frac{\epsilon_{n}}{2}\big) - \big(\epsilon_{i_1} + \cdots \epsilon_{i_k}\big),
\]
where $i_1 < \dots < i_k$.
\end{remark}

Relations between generators of $Cl(2n)$ are even, so $Cl(2n)$ admits a decomposition into even and odd subspaces $Cl(2n)=Cl(2n)^+\oplus Cl(2n)^-$, as a module over the subalgebra $Cl(2n)^+$. There is a respective decomposition of the spin representation $S=S^+\oplus S^-$, as $Cl(2n)^+$-modules, which corresponds to even and odd subspaces in $S$.

\begin{corollary}
    Restricting to $\son[2n]$, the spin module $S$ admits a decomposition into two irreducible $\son$-modules $S^+$ and $S^-$. We have the isomorphisms of $\son$-modules
    \begin{equation}\label{eq:Spmisohwt}
    S^+ \cong V_{\varpi_n}\quad \text{and} \quad S^-\cong V_{\varpi_{n-1}}.
    \end{equation}
\end{corollary}
\begin{proof}
    Note that in Lemma \ref{lem:so2ntocliff}, the elements $M_{i,j}$ map to elements in $Cl(2n)^+$, so $S=S^+\oplus S^-$. The isomorphisms \eqref{eq:Spmisohwt} follow from weight space decompositions.
\end{proof}

Now we consider $S\ot S$ over $\son[2n]$ and record the following for later use.

\begin{proposition}\label{FH_decomposition}
For $n\in \mathbb{Z}$, let $p(n)\in \{0,1\}$ be the remainder of $n$ after dividing by $2$, i.e., the parity of $n$. We have the following decomposition of $\son$-modules:
\begin{align*}
    S^+\ot S^+ &\cong V_{2\varpi_n} \oplus V_{\varpi_{n-2}} \oplus V_{\varpi_{n-4}}\oplus \dots \oplus V_{\varpi_{p(n)}}, \\
    S^-\ot S^- &\cong V_{2\varpi_{n-1}} \oplus V_{\varpi_{n-2}} \oplus V_{\varpi_{n-4}}\oplus \dots \oplus V_{\varpi_{p(n)}}, \quad \text{and} \\
    S^-\otimes S^+ \cong S^+\ot S^- &\cong V_{\varpi_{n-1}+\varpi_{n}} \oplus V_{\varpi_{n-3}} \oplus \dots  \oplus V_{\varpi_{1-{p}(n)}},
\end{align*}
where $\varpi_0:=0$.
\end{proposition}
\begin{proof}
    See \cite[Exercise 23.31(b)]{FultonHarris}.
\end{proof}

\subsection{Classical Howe duality for $(\son[2n], O_m)$}

We now work with the Clifford algebra $Cl(2nm)$, which by Lemma \ref{lem:cliffordgensrelns} has generators
\[
\psi_{ia},\psi_{ia}^{*}, \quad 1\le i \le n, \quad 1\le a \le m,
\]
and relations
\begin{equation}\label{eq:cliffrelns}
\psi_{ia}\psi_{jb}^* +\psi_{jb}^*\psi_{ia} = \delta_{ij}\delta_{ab}, \quad \psi_{ia}\psi_{jb}=-\psi_{jb}\psi_{ia}, \quad \text{and} \quad \psi_{ia}^*\psi_{jb}^*=-\psi_{jb}^*\psi_{ia}^*.
\end{equation}
We write $\howespin$ to denote the spin module over $Cl(2nm)$, see Definition \ref{def:spinmodule}. 

Consider the following Clifford subalgebras of $Cl(2nm)$
\[
Cl_n^{(a)}=\langle \psi_{ia}, \psi_{ia}^* \ | \ 1\le i \le n \rangle, \quad \text{for $1\le a \le m$}.
\] 
Note that
\[
Cl_n^{(1)}\otimes \dots\otimes  Cl_n^{(m)}\hookrightarrow Cl(2nm),
\]
and for each $1\le a\le m$, we also have $\son^{(a)}\hookrightarrow Cl_n^{(a)}$, see Lemma \ref{lem:so2ntocliff}. 

\begin{lemma}\label{lem:Sisso2nmodule}
There is an action of $\son$ on $\howespin$ via
\[
\son \hookrightarrow\son^{(1)}\otimes \dots \otimes \son^{(m)} \hookrightarrow Cl(2nm),
\qquad 
x \mapsto \sum_{1\le a \le m} {x}^{(a)} =: {x}^{\langle n \rangle}.
\]
Under this embedding we have $\howespin\cong S^{\ot m}$, where each $\son^{(a)}$ acts on the respective tensor component.
\end{lemma}

\begin{notation}\label{not:psibasis}
    If $I\subset \{1, \dots, n\}\times \{1, \dots, m\}$, then write $I_a:=\{i \ | \ (i,a)\in I\}$, for $a=1, \dots, m$. Define elements in $\howespin$:
    \[
    \hwp_{I}:=\prod^{\rightarrow}_{i_1\in I_1}\psi_{i_11}^* \dots  \prod^{\rightarrow}_{i_m\in I_m}\psi_{i_mm}^*.
    \]
    The isomorphism $\howespin\cong S^{\otimes m}$ is such that $\hwp_{I}\mapsto \hwp_{I_1}\otimes \dots \otimes \hwp_{I_m}$.
\end{notation}

\begin{example}
    If $I=\{(1,1), (2,1), (4,1), (1,3),(3,3)\}\subset \{1,2,3,4\}\times \{1,2,3\}$, then $I_1=\{1,2,4\}$, $I_2= \emptyset$, and $I_3 = \{1,3\}$. Thus, $\hwp_I= \psi_{11}^*\psi_{21}^*\psi_{41}^*\psi_{13}^*\psi_{33}^*$.
\end{example}

An $O_m$-module can be reconstructed from the induced $\som$ action, together with the action of the diagonal element $e^{\pi{\bf i}E_{1,1}}={\rm diag}(-1,1,\dots)\in O_m$. Thus, to make the action of $O_m$ on $\howespin$ explicit, we will describe the action of $\som$ and the action of $e^{\pi{\bf i}E_{1,1}}$.

To describe the action of $\som$, consider the Clifford subalgebras of $Cl(2nm)$
\[
{}^{(i)}Cl_m:=\langle \psi_{ia}, \psi_{ia}^* \ | \ 1\le a \le m \rangle, \quad \text{for $1\le i\le n$}.
\]
For each $i$, there is a Lie subalgebra
\[
{}^{(i)}\mathfrak{so}_m \hookrightarrow {}^{(i)}Cl_m,
\]
such that the basis of $\som$ maps to $Cl(2nm)$ via 
\begin{equation}\label{eq:Babactionsom}
B_{ab}:=E_{ab}-E_{ba}\mapsto \frac{[\psi_{ia}^*,\psi_{ib}]}{2}-\frac{[\psi_{ib}^*,\psi_{ia}]}{2}  \stackrel{\eqref{eq:cliffrelns}}{=} \psi_{ia}^*\psi_{ib}-\psi_{ib}^*\psi_{ia}=:{}^{(i)}B_{ab}.
\end{equation}

\begin{lemma}\label{lemOmaction}
The action of $O_m$ on $\howespin$ is determined as follows. The action of $B_{ab}\in \mathfrak{so}_m$ on $\howespin\cong S^{\otimes m}$ is by $\sum_{i=1}^n{}^{(i)}B_{ab}$. The action of the diagonal element $e^{\pi{\bf i}E_{1,1}}\in O_m$ is given by
\[
e^{\pi{\bf i}E_{1,1}} \cdot \hwp_{I}=(-1)^{|I_1|}\hwp_I,
\]
where $I\subset \{1, \dots, n\}\times \{1, \dots, m\}$.
\end{lemma}

The following is a classical result, see \cite[Section 4.3.5]{How95} and \cite[Proposition 5.36]{CW12}.

\begin{theorem}\label{Orthogonal_Howe_duality}
The actions of $\son$ and $O_m$ on $\howespin$ commute and generate each other's centralizers. Under the action of the pair $(\son,O_m)$, the module $\howespin$ has the following multiplicity-free decomposition
\begin{equation}\label{eq:classicalmultfreedecomposition}
    \howespin\cong \bigoplus_{\substack{\lambda=(\lambda_1,\dots, \lambda_m) \\ \lambda^{\prime}_1+\lambda^{\prime}_2\le m, \ l(\lambda^{\prime})\le n}} V^{\son}(\lambda^{\top})\ot V^{O_m}(\lambda),
\end{equation}
where 
\begin{itemize}
    \item $\lambda$ is a partition (of length not greater than $m$),
    \item $\lambda'$ is the transposed partition,
    \item $l(-)$ is the length of a partition,
    \item $\lambda^\top :=\sum_{i=1}^{n}(\frac{m}{2}-\lambda'_{n-i+1})\epsilon_i$ is an $\son$ weight (note that $\lambda^{\top}_1\ge \dots \ge \lambda_{n-1}^{\top}\ge |\lambda_n^{\top}|$),
    \item $V^{\son}(\lambda^{\top})$ is the (unique) irreducible $\son$-representation of highest weight $\lambda^\top$, and
    \item $V^{O_m}(\lambda)$ is the (unique) irreducible $O_m$-representation labelled by $\lambda$ (see \cite[Section 19.5]{FultonHarris}).
\end{itemize}
\end{theorem}

In particular, the $O_m$ action on $\howespin$ generates $\End_{\son}(\howespin)$ and $V^{O_m}(\lambda)\cong \Hom_{\son}(V^{\son}(\lambda^{\top}), \howespin)$.

\begin{notation}
    Write $\flip_{a,a+1}\in \End_{\son}(\howespin)$ to denote the endomorphism corresponding, under $\howespin\cong S^{\otimes m}$, to $\id_{S^{\otimes a-1}}\otimes \flip \otimes \id_{S^{\otimes m-(a+1)}}$, where $\flip(v\otimes w) = w\otimes v$. 
\end{notation}

We will describe the endomorphism $\flip_{a,a+1}$ in terms of the $O_m$ action on $\howespin$. Note that the permutation matrices in $GL_m$ are elements of $O_m$. Write $(a,a+1)\in O_m\subset GL_m$ to denote the permutation matrix corresponding to a simple transposition.

\begin{lemma}\label{lem:signforaap1}
    The action of permutation matrices $ S_m\subset O_m$ on $\howespin$ agrees, up to sign, with the action of permutation of tensor factors on $\howespin\cong S^{\otimes m}$. Namely, 
    \[
    (a,a+1)\cdot \hwp_I=(-1)^{\#I_a\#I_{a+1}}\flip_{a,a+1}\cdot\hwp_I.
    \]
\end{lemma}

\begin{proof}
   See Appendix \ref{proof3}.
\end{proof}

\begin{remark}\label{rem:stilde(1,2)}
    The matrix $(a,a+1)$ is not an element of $SO_m$, while the element $\tilde{s}_a\in SO_m$ from Lemma \ref{lift_of_braid_gp_gens} is. However, the difference between $(a,a+1)$ and $\tilde{s}_a$ is easily expressed by the equalities 
    \[
    \tilde{s}_a = (a,a+1)e^{\pi{\bf i} E_{a+1,a+1}} =e^{\pi{\bf i} E_{a,a}}(a,a+1).
    \]
    This difference also appears in $(GL_n,GL_m)$ duality, when passing to $SL_m$, see \cite[Corollary 3.6]{Valerio-qWeyl}. Note that $e^{\pi{\bf i} E_{a+1,a+1}}\hwp_I = (-1)^{\# I_{a+1}}\hwp_I$, so $\tilde{s}_a\cdot \hwp_I = (-1)^{\#I_{a+1}}(-1)^{\#I_a\#I_{a+1}}\flip_{a,a+1}\hwp_I$.
\end{remark}

\subsection{Relating connections under classical $(\son, O_m)$ Howe duality}

Write $\epsilon_a\in \mathfrak{h}_{\glm}^*$, for $\mathfrak{h}_{\glm}\subset \glm$, to denote the functional such that $\epsilon_a(E_{bb})=\delta_{ab}$. Using that $Conf_m(\C) \subset \C^m$ and $\mathfrak{h}_{\glm}^{reg}\subset \mathfrak{h}_{\glm}$, we identify
\begin{equation}\label{eq:identityconfhreg}
Conf_m(\C)\rightarrow \mathfrak{h}_{\glm}^{reg}, \quad (z_1, \dots, z_m) \mapsto \mathrm{diag} (z_1, \dots, z_m).
\end{equation}
Under \eqref{eq:identityconfhreg}, the symmetric group action on $Conf_m(\C)$ identifies with the Weyl group action on $\mathfrak{h}_{\glm}^{reg}$, see Remark \ref{rem:Wactiononh-permutescoords}.

Consider the trivial bundle with fiber $\howespin$ over $Conf_m(\C)$. If we treat $\howespin$ as an $\son$-module as in Lemma \ref{lem:Sisso2nmodule}, then $\howespin\cong S^{\ot m}$, so we have the KZ connection on $Conf_m(\C)\times \howespin$:
\begin{equation}\label{KZ_so2n}
    \nabla_{KZ,\son}^{h,\langle n\rangle}= d-h\sum_{1\le a< b \le m} \frac{dz_a -dz_b}{z_a-z_b}2\cdot \Omega_{ab}^{\langle n \rangle}, \quad \Omega_{ab}^{\langle n \rangle}:=\sum (M_{ij})^{(a)}\ot (M^{ij})^{(b)}.
\end{equation}
In the equality defining $\Omega_{ab}^{\langle n\rangle}$, the sum is over the basis $M_{ij}$ of $\son$, see \eqref{eq:basisso2ncartan} and \eqref{eq:basisso2nroots}, where the dual basis elements $M^{ij}$ are as in Remark \ref{rem:dualmij}. The action of $M^{(c)}_{ij}$ on $\howespin\cong S^{\otimes m}$ is by $\psi_{ic}\psi_{-jc} - \frac{1}{2}\delta_{i,j}$, where  if $j<0$, then we interpret $\psi_{-jc}:=\psi_{jc}^*$, see Lemma \ref{lem:so2ntocliff}.

\begin{remark}
    Here, we choose to work with the split Casimir operator $2\cdot \Omega_{ab}^{\langle n\rangle}\in U(\son)\otimes U(\son)$ so that the corresponding Casimir operator in $U(\son)$ acts by $(\lambda,\lambda+2\rho)$ on the highest weight representation $V^{\son}(\lambda)$ (where $(\epsilon_i,\epsilon_j)=\delta_{ij}$ and $\rho=(n-1)\cdot\epsilon_1+\dots+1\cdot\epsilon_{n-1}+0\cdot \epsilon_n$). For example, $2\cdot\sum M_{ij} M^{ij}$ acts by $2n-1$ on the tautological representation $\C^{2n}$.
\end{remark}

On the other hand, if we treat $\howespin$ as an $\mathfrak{so}_m$ module as in Lemma \ref{lemOmaction}, and use the identification in \eqref{eq:identityconfhreg}, then we may also define the boundary Casimir connection on $Conf_m(\C)\times \howespin$:
\begin{equation}\label{bCas_extended_som}
    {}^{\langle m\rangle}\nabla_{bCas,\mathfrak{so}_m\subset \mathfrak{gl}_m}^{h}= d-h\sum_{1\le a< b \le m} \frac{dz_a-dz_b}{z_a-z_b}{}^{\langle m\rangle}B^2_{ab}, \quad {}^{\langle m\rangle}B^2_{ab}:=\left(\sum_{i=1}^n{}^{(i)}B_{ab}\right)^2.
\end{equation}
The action of ${}^{(i)}B_{ab}$ on $\howespin$ is given by $\psi_{ia}^*\psi_{ib}-\psi_{ib}^*\psi_{ia}$, see \eqref{eq:Babactionsom}.

\begin{proposition}[{\cite[Theorem 10.9]{MR4983792}}]\label{Gauge_equiv}
The KZ connection and the boundary Casimir connection are related on $\howespin$ as follows:
\begin{equation}\label{KZ-bCas}
    \nabla_{KZ,\son}^{h, \langle n \rangle}+\frac{nh}{4}\sum_{1\le a<b\le m}\frac{dz_a-dz_b}{z_a-z_b}={}^{\langle m\rangle}\nabla_{bCas,\mathfrak{so}_m\subset \mathfrak{gl}_m}^{h/2}.
\end{equation}
\end{proposition}
\begin{proof}[Proof sketch]
    For $1\le a< b\le m$, compare $\Omega_{ab}^{\langle n\rangle}$ to ${}^{\langle m\rangle}B_{ab}^2$ as endomorphisms of $\howespin$, i.e., compare
    \[
    \sum (\psi_{ia}\psi_{-ja} - \frac{1}{2}\delta_{ij})(\psi_{jb}\psi_{-ib} - \frac{1}{2}\delta_{ji}) \quad \text{to} \quad \frac{1}{2}\left(\sum_{i=1}^n\psi_{ia}^*\psi_{ib}-\psi_{ib}^*\psi_{ia}\right)^2.
    \]
    Expanding and using Clifford relations, we find the operator $2\cdot \Omega_{ab}^{\langle n \rangle} - \frac{1}{2}{}^{\langle m\rangle}B_{ab}^2$ acts on $\howespin$ by $\frac{n}{4}\cdot \id_{\howespin}$.
\end{proof}

\begin{lemma}
    Let $M\cong V^{O(m)}(\lambda)$ be an irreducible $O(m)$ module appearing in the decomposition of $\howespin$ from Theorem \ref{Orthogonal_Howe_duality}. There is an embedding $M\hookrightarrow \mathbb{S}$ such that both the KZ connection and the boundary Casimir connection admit fiber-wise (equivariant) restrictions to $M$.
\end{lemma} 

\begin{proof}
    If $M\hookrightarrow \mathbb{S}$ is an $\som$-equivariant embedding, then the equality \eqref{KZ-bCas} shows the operators $\Omega_{ab}^{\langle n \rangle}$ and ${}^{\langle m\rangle}B^2_{ab}$ naturally act on $M$, and Lemma \ref{lift_of_braid_gp_gens} shows that the operator $\widetilde{s}_a$ naturally acts on $M$.

It remains to ensure that the operators $\flip_{a,a+1}\in \End(\howespin)$ act on $M$, via $M\hookrightarrow \howespin$, as well. By Theorem \ref{Orthogonal_Howe_duality}, we have a surjection $U(\son)\rightarrow \End_{O(m)}(\howespin)$, and the $M$-isotypic component of $\howespin$ is identified with $V^{\son}(\lambda^{\top})\cong \Hom_{O(m)}(M, \howespin)$. Let us choose an element $p\in U(\son)$ which projects onto a non-zero vector from $V^{\son}(\lambda^{\top})$ in the isotypic component corresponding to $M$ and acts by $0$ on any other isotypic component appearing in the decomposition of $\mathbb{S}$. Since $p\in U(\son)$, it follows from Theorem \ref{Orthogonal_Howe_duality} that the element $e^{{\rm i}\pi E_{1,1}}\in O(m)$ commutes with $p$. By Lemma \ref{lem:signforaap1}, the action by the operators $\flip_{a,a+1}\in \End(\howespin)$ restricts to $M$ along the embedding $M\hookrightarrow \mathbb{S}$ via $p$.
\end{proof}

Observe that we have an $S_m$-equivariant map
\[
\pi_{S_m}: Conf_m(\C)\twoheadrightarrow UConf_m(\C).
\]
Since the KZ connection is $S_m$-equivariant as well, its pushforward $\pi_{S_m*}(\nabla_{KZ,\son}^{h,\langle n\rangle})$ is well-defined. Furthermore, by considering the decomposition $\mathfrak{h}_{\glm}=\mathfrak{h}_{\slm}\oplus \C(1,\dots,1)$ and an $S_m$-equivariant identification $Conf_m(\C) \cong \mathfrak{h}_{\glm}^{reg}$ we can see that we have an $S_m$-equivariant sequence of isomorphisms
\[
UConf_m(\C)\cong \mathfrak{h}_{\glm}^{reg}/S_m = \mathfrak{h}_{\slm}^{reg}/ S_m \times \C(1,\dots, 1).
\]
Note that the KZ connection is translation-invariant with respect to the direction $(1,\dots,1)$. It follows that $\pi_{S_m*}(\nabla_{KZ,\son}^{h,\langle n\rangle})$ is trivial along the same direction on $\mathfrak{h}_{\slm}^{reg}/ S_m \times \C(1,\dots, 1)$, and the latter space deformation retracts onto the former. Therefore, if we consider the embedding of the $0$ fiber $\mathfrak{h}_{\slm}^{reg}/S_m$ under the map
\[
i_0:\mathfrak{h}_{\slm}^{reg}/S_m \hookrightarrow \mathfrak{h}_{\slm}^{reg}/S_m \times \C(1,\dots, 1),
\]
we can work with the flat connection $i_0^*\pi_{S_m*}(\nabla_{KZ,\son}^{h,\langle n\rangle})$ with logarithmic singularities at $\alpha\in \Phi_+$ (for the corresponding root system of $\slm$), and the monodromy representations for both connections are the same.

Now we note that $P:=\prod_{1\le a<b\le m}(z_a-z_b)^2$ is a well-defined regular function on $\mathfrak{h}_{\slm}^{reg}/S_m$. From Proposition \ref{Gauge_equiv} we have the following.

\begin{proposition}\label{KZ_bCas_flat_sections}
    A map $f:\mathfrak{h}_{\slm}^{reg}/S_m \rightarrow \mathbb{S}$ is a flat section of the KZ connection $i_0^*\pi_{S_m*}(\nabla_{KZ,\son}^{h,\langle n\rangle})$ if and only if $g=f\cdot P^{-\frac{nh}{8}}$ is a flat section of $\nabla_{bCas,\som\subset\slm}^{h/2}$.
\end{proposition}

\begin{corollary}\label{Br_S_m_monodromy}
    Let $M\subset \howespin\cong S^{\otimes m}$ be an $O(m)$-submodule.
    Consider the monodromy representations
    \begin{equation}
        \pi_{KZ,\son}^h, \pi_{bCas,\som\subset \slm}^{h/2}: {\rm Br}_{S_m} \rightarrow GL(M)
    \end{equation}
    associated with the restriction to $M$ of the connections \eqref{KZ_so2n} and \eqref{bCas_extended_som}, respectively. Then we have
\[
\pi_{bCas, \som\subset \slm}^{h/2}(\sigma_a)= e^{\pi {\bf i}E_{a,a}} (a,a+1)\flip^{-1}_{a,a+1} \pi_{KZ, \son}^h(\sigma_a)e^{-\frac{\pi \imag n h}{4}},
\]
    where $\sigma_a$ are the braid generators of ${\rm Br}_{S_m}$ associated with the hyperplanes $z_a=z_{a+1}$ of the covering space $\mathfrak{h}_{\slm}^{reg}$.
\end{corollary}

\begin{proof}
    Using Proposition \ref{KZ_bCas_flat_sections}, we obtain the desired relation between the monodromy representation for the boundary Casimir connection, see Proposition \ref{parralel_transport}, and the monodromy representation for the KZ connection, see Proposition \ref{parralel_transport_KZ}.
\end{proof}


\section{Quantum groups and quantum symmetric pairs}\label{sec4}
\subsection{Quantum groups}

Let $h$ be a generic parameter in $\C$. Fix a simple complex Lie algebra $\g$, and choose $\h\subset \mathfrak{b}\subset \g$, and $\Delta \subset \Phi_+\subset \Phi$, as in \S \ref{sec:lietheorynotation}. 
\begin{notation}\label{not:quantumintegernotation}
We will consider the following integers and quantum integers.
\begin{itemize}
    \item $d_{\alpha_i}\in \{1, 2, 3\}$ denotes the root lengths, for $\alpha_i\in \Delta$.
    \item $a_{ij}:=2\frac{(\alpha_i,\alpha_j)}{(\alpha_i,\alpha_i)}$ denotes the entries of the Cartan matrix of $\mathfrak{g}$, for $\alpha_i,\alpha_j\in \Delta$.
    \item $q:=e^h$ and $q_i:=e^{d_{\alpha_i}h}$.
    \item $[n]_i:=\frac{q^n_i-q^{-n}_i}{q_i-q^{-1}_i}$, 
    $[n]_i!:=[n]_i \cdot \dots \cdot [1]_i$, and $\qbinom{n}{k}_i:=\frac{[n]_i!}{[k]_i! [n-k]_i!}$.
\end{itemize} 
\end{notation}

\begin{definition}[\cite{JantzenQgps,Lus4}]\label{Quantum_group}
    The \emph{quantum group} $U_{h}(\mathfrak{g})$ is an algebra generated by elements $E_i,F_i, K_i^{\pm 1}, 1\le i \le |\Delta|$, with the following relations
    \[
    K_iK_j=K_jK_i, \quad K_iK_i^{-1} = 1= K_{i}^{-1}K_i, \quad K_i E_j K_i^{-1}=q^{a_{ij}}_i E_j, \quad K_i F_j K_i^{-1} = q^{-a_{ij}}_iF_j,
    \]
    \[
    E_iF_j-F_jE_i=\delta_{ij}\frac{K_i-K_i^{-1}}{q_i-q_i^{-1}},
    \]
    \[
    \sum_{k=0}^{1-a_{ij}}(-1)^k \qbinom{1-a_{ij}}{k}_i E_i^{1-a_{ij}-k}E_j E_i^k=0,
    \]
    \[
    \sum_{k=0}^{1-a_{ij}}(-1)^k \qbinom{1-a_{ij}}{k}_i F_i^{1-a_{ij}-k}F_j F_i^k=0.
    \]
    Additionally, we will utilize the unique Hopf algebra structure on $U_h(\mathfrak{g})$ such that
    \begin{equation*}
        \Delta(E_i)=E_i \ot 1+ K_i\ot E_i, \quad \ep(E_i)=0, \quad S(E_i)=-K_i^{-1}E_i,
    \end{equation*}
    \begin{equation*}
        \Delta(F_i)=F_i \ot K_i^{-1}+ 1\ot F_i, \quad \ep(F_i)=0, \quad S(F_i)=-F_iK_i,
    \end{equation*}
    \begin{equation*}
        \Delta(K_i)=K_i\otimes K_i, \quad \ep(K_i)=1, \quad S(K_i)=K_i^{-1}.
    \end{equation*}
     For a coroot lattice element $\mu=\sum_{i=1}^{|\Delta|}m_i \alpha_i^{\vee}\in Y$, we define an element $K_{\mu}:=\prod_{i=1}^{|\Delta|} K_{i}^{m_i} \in U_{h}(\g)$.
\end{definition}

\subsection{Integrable modules, Lusztig's braid symmetries, and $R$-matrix}

\begin{definition}
    We call a $U_h(\mathfrak{g})$-module $M$ \emph{type I} if $M=\oplus_{\lambda\in X} M[\lambda]$, where $M[\lambda]$ are such that if $v\in M[\lambda]$, then $v$ is a weight vector of weight $\lambda$, i.e., $K_i v=q_i^{\langle \alpha_i^{\vee},\lambda\rangle}v$, for $1\le i \le |\Delta|$. 
\end{definition}

\begin{definition}
    We call $M$ \emph{integrable} if $E_i$ and $F_i$ act locally nilpotently on $M$, for $1\le i\le |\Delta|$.
\end{definition}

In this paper we will work with the following category.

\begin{definition}\label{fd-Uq-mod}
Let $\mathcal{C}_{\mathfrak{g}}$ be the category of type $I$ finite-dimensional integrable $U_h(\mathfrak{g})$-modules $M$.
\end{definition}

\begin{remark}\label{rem:quantumgroupmonoidal}
Using the coproduct $\Delta$ from Definition \ref{Quantum_group}, and the trivial associator $(V\otimes W)\otimes U\rightarrow V\otimes (W\otimes U)$, denoted $1$, we can view $\mathcal{C}_{\g}$ as a monoidal category.
\end{remark}

We also recall the definition of Lusztig's symmetries \cite{Lus4}. 

\begin{definition}
    The \emph{divided powers} of $E_i$ and $F_i$ are defined to be
    \[
    E_i^{(a)}:=\frac{E_i^a}{[a]_i!} \quad \text{and} \quad F_i^{(a)}:=\frac{F_i^a}{[a]_i!},
    \]
    respectively.
\end{definition}

\begin{definition}\label{def:Lussymmetry}
    Let $V\in \mathcal{C}_{\g}$. For $1\le i\le |\Delta|$, define \emph{Lusztig's symmetry} $T_i\in \End(V)$ as follows:
    \begin{equation*}
        T_i\cdot v=\sum_{\substack{a,b\ge 0 \\ b-a=\langle\alpha_i^{\vee},\lambda\rangle}}(-q_i)^bE_i^{(a)}F_i^{(b)}v, \qquad \text{for all $v\in V[\lambda]$}.
    \end{equation*}
See \cite[Remark 2.1]{Cautis} for this simplified formula. Our $T_i$ corresponds to $T''_{1,i}$ in Lusztig's notation \cite{Lus4}. 
\end{definition}

The operators $T_i$ are invertible and satisfy the braid relations of the Weyl group $W$ of $\mathfrak{g}$. Thus, for $w\in W$ and any reduced decomposition
\begin{equation*}
    w=s_{i_1}\cdots s_{i_k},
\end{equation*}
we can unambiguously write
\begin{equation*}
    T_w:=T_{i_1}\cdots T_{i_k}.
\end{equation*}
We will use a description of the $R$-matrix in terms of $T_{w_0}$, where $w_0\in W$ is the longest element, which originates in \cite{KirResh, MR1142216}. The following formulation is from {\cite[Theorem 7.1]{KamnTin}}.

\begin{theorem}\label{R-matrix_braiding}
    The monoidal category $\mathcal{C}_{\mathfrak{g}}$ admits a braiding $R_{V,W}$, for $V,W\in \mathcal{C}_{\mathfrak{g}}$, such that
\begin{equation}\label{R-matrix}
     R|_{V\ot W}:= \flip|_{V\ot W} \circ q^{(\wt(-),\wt(-))}\circ (T_{w_0}^{-1} \ot T_{w_0}^{-1} 
 )\circ \Delta(T_{w_0}).
\end{equation}
Here $\flip$ is the map
\begin{equation*}
    \flip: V\ot W \rightarrow W\ot V, \quad v\otimes w\mapsto w\otimes v,
\end{equation*}
and
\[
q^{(\wt(-),\wt(-))}:V[\lambda]\otimes W[\mu]\rightarrow V[\lambda]\otimes W[\mu], \quad v\otimes w\mapsto q^{(\lambda,\mu)}v\otimes w,
\]
where $(\lambda,\mu):=\sum_{1\le i,j\le |\Delta|} B_{ij}^{-1} \langle \alpha_i^\vee, \lambda \rangle \langle \alpha_j^\vee, \mu \rangle$ and $B_{ij}:=a_{ij}d_j^{-1}$.
\end{theorem}

\begin{notation}
    For $V\in \mathcal{C}_{\g}$, we have a group homomorphism
    \[
    \pi_{R,\g}^h:\Br_{S_m}\rightarrow GL(V^{\otimes m}),
    \]
    defined by $\pi_{R,\g}^h(\sigma_a):=\id_{V^{\otimes a-1}}\otimes R_{V,V}\otimes \id_{V^{\otimes m-a-1}}$.
\end{notation}

\begin{remark}\label{Hopf_quasitriangular}
    In Definition \ref{Quantum_group}, we use the standard Lusztig's coproduct $\Delta$. The coproduct in \cite{KamnTin} is $\Delta^{op}$. Generally, for a quasi-triangular Hopf algebra $(H,\Delta,\epsilon,S,\widetilde{R})$, the universal $R$-matrix $\widetilde{R}$ for $\Delta$ and the universal $R$-matrix $\widetilde{R}^{op}$ for $\Delta^{op}$ are related as $\widetilde{R}^{op}=\widetilde{R}_{21}$ (one should distinguish the universal $R$-matrix $\widetilde{R}$ and the braiding $R=\flip \circ \widetilde{R}$). Since the operators $q^{(\wt(-),\wt(-))}$ and $T^{-1}_{w_0}\ot T^{-1}_{w_0}$ commute with transposition of factors, it follows that the universal $R$-matrix (and the braiding) for our choice of coproduct, can be written in the same form as in \cite[Theorem 7.1]{KamnTin}.
    
    Furthermore, let us denote by $U_{h}^{\ge 0}(\mathfrak{g})$ and $U_{h}^{\le 0}(\mathfrak{g})$ the subalgebras of $U_h(\mathfrak{g})$ corresponding to the Borel subalgebra $\mathfrak{b}\subset \mathfrak{g}$ and its opposite respectively, then according to \cite[Proposition 8.3.13]{CP} any universal $R$-matrix such that $\widetilde{R}\in U_{h}^{\le 0}(\mathfrak{g})\hat{\ot} U_{h}^{\ge 0}(\mathfrak{g})$ in an appropriate completion of the tensor product is unique up to a scalar. The scalar can be fixed via consideration of the counit.
\end{remark}

The following Theorem is a known result (e.g., see \cite[Theorem 3.1]{NeTu10}).

\begin{theorem}\label{thm:equivalencequantumclassical}
    For generic $h\in \C$ there is an equivalence of the braided tensor categories \[
    ({\rm Rep}(\mathfrak{g}), \Delta_{cl},\Phi_{h,KZ},e^{\pi {\rm i}h\Omega})\cong(\mathcal{C}_{\mathfrak{g}},\Delta, 1, R)_{\pi {\rm i} h},
    \]
    where $\Delta_{cl}$ is the classical cocommutative coproduct on $U(\mathfrak{g})$, and $\Phi_{KZ}$ is the Drinfeld associator arising from the KZ connection \eqref{KZ_defn} with the split Casimir $\Omega$. The category $\mathcal{C}_{\mathfrak{g}}$ and the braiding $R$ are as in the Definition \ref{fd-Uq-mod}, Remark \ref{rem:quantumgroupmonoidal}, and Theorem \ref{R-matrix_braiding}, but with $q=e^{\pi {\rm i}h}$ (as opposed to $q=e^{h}$ as in Notation \ref{not:quantumintegernotation}).
\end{theorem}

\begin{remark}\label{Artin's_lifting}
    From the proof of this Theorem it follows that for any finite collection of finite-dimensional $\mathfrak{g}$-modules $V_i,i\in I$ the Drinfeld's associator $\Phi_{h,KZ}$ is analytic (and holomorphic) in a neighborhood of $h=0$ for each triple $V_i \ot V_j \ot V_k$. For such a collection of $V_i$ an equivalence can be realized via an invertible Drinfeld twist $F_h|_{V_i,V_j} \in {\rm Im}(\rho_{V_i}\ot \rho_{V_j})$, where $\rho_{V_i}\ot \rho_{V_j}: U(\mathfrak{g})\ot U(\mathfrak{g}) \rightarrow \End(V_i\ot V_j)$, which is analytic (more precisely, holomorphic) in a neighborhood of $h=0$. Indeed, it is known that such a twist exists over a ring of power series $\C[[h]]$ and we have a solution $F_0=Id$ at $h=0$. Therefore, by Artin's analytic approximation Theorem there must be a locally holomorphic solution for $F_h$. For our purposes we will work with $V_i=S$. We can also extend such a twist $F_h$ to the case of $S^{\ot m}$ which we denote by $F^{(m)}_h$, in particular $F_h=F_h^{(2)}$.  
\end{remark}


\subsection{Quantum symmetric pairs}

Let $\tau$ be an involution on $\Delta$ which extends to involutions on $\mathfrak{h}$ and $\mathfrak{h}^*$, preserving the pairing $\langle h,\lambda\rangle =\lambda(h)$. Fix $\Delta_{\bullet} \subset \Delta$ of finite type. Write $W_{\Delta_{\bullet}}\subset W$ for the associated Weyl subgroup with $w_{\bullet}$ the longest element. Let $\rho_{\bullet}^{\vee}$ denote the half sum of $\Phi_{\bullet}$'s positive coroots.

\begin{definition}[\cite{W22}]\label{Satake_diagram}
    A \emph{Satake diagram} $(\Delta=\Delta_{\bullet}\cup \Delta_{\circ}, \tau)$, with $\Delta_{\circ}=\Delta\setminus \Delta_{\bullet}$, is the data as above such that $\tau(\Delta_{\bullet})=\Delta_{\bullet}$, the actions of $\tau$ and $-w_{\bullet}$ on $\Delta_{\bullet}$ coincide, and $\alpha_i (\rho_{\bullet}^{\vee}) \in \mathbb{Z}$ whenever $\tau(i)=i\in\Delta_{\circ}$.
\end{definition}
We denote by $\theta$ the involution of $\mathfrak{h},\mathfrak{h}^*$ given by $\theta=-w_{\bullet}\circ \tau$.

\begin{definition}
    If $\Delta_{\bullet}= \emptyset$, then we say the $\iota$quantum group is \emph{quasi-split}. If $\Delta_{\bullet}=\emptyset$ and $\tau=\id$, then we say the $\iota$quantum group is \emph{split}.
\end{definition}

\begin{remark}
    In the split case, the involution $\theta:\mathfrak{g}\rightarrow \mathfrak{g}$ associated to the Satake diagram is the Chevalley automorphism \eqref{eq:Chevalleyauto}.  
\end{remark}

We define the $\iota$coroot lattice as
\begin{equation*}
    Y^{\iota}:=\{\mu\in Y|\theta(\mu)=\mu\}
\end{equation*}
and the $\iota$weight lattice
\[
X_{\iota}:=X/\breve{X}, \quad \breve{X}:=\{\lambda-\lambda^{\theta}|\lambda\in X\}.
\]
There is a canonical pairing
\[
\langle-,-\rangle: Y^{\iota} \times X_{\iota} \rightarrow \mathbb{Z},
\]
but this pairing will be degenerate in general.

\begin{example}
    Let $\Delta$ be the set of simple roots for simple $\mathfrak{g}$ and consider the corresponding split Satake diagram (so $\Delta_{\circ} = \Delta$, $\Delta_{\bullet}= \emptyset$, and $\tau= \id$). Then $Y^{\iota} = 0$ and $X_{\iota}\cong \oplus_{\alpha_i\in \Delta}\mathbb{Z}/2\cdot \overline{\varpi_i}$.
\end{example}

\begin{definition}\label{iQuantumgroup}
    Fix a Satake diagram $(\Delta=\Delta_{\bullet}\cup\Delta_{\circ}, \tau)$. Choose scalars $\varsigma_i\in \C^{\times}$ and $\kappa_i\in \C$ satisfying the conditions
    \begin{itemize}
        \item $\varsigma_i=\varsigma_{\tau(i)}$, if $(\alpha_i,\alpha_{\tau(i)})=0$. 
        \item $\kappa_i=0$, unless $\tau(i) =i$, and $(\alpha_i,\alpha_j)=0$ for all $\alpha_j\in \Delta_{\bullet}$.
        \item $\alpha_i(h_{\alpha_k})\in2\mathbb{Z}$, for all $\alpha_k\in \Delta_{\circ}$, such that $\tau(\alpha_i) = \alpha_j= w_{\bullet}(\alpha_k)$.
    \end{itemize}
    Let $\mathfrak{g}$ be the Lie algebra associated with $\Delta$. The \emph{$\iota$quantum group} 
    \[
    U_h^{\iota}:=U^{\iota}_{h}(\Delta=\Delta_{\bullet} \cup \Delta_{\circ}, \tau,\varsigma,\kappa)
    \]
    is the subalgebra of $U_{h}(\mathfrak{g})$ generated by the following elements:
    \begin{equation*}
        B_i:= F_i +\varsigma_i T_{w_{\bullet}}(E_{\tau(i)})K_i^{-1} + \kappa_i K_i^{-1} \ (\alpha_i\in \Delta_{\circ}), \quad K_{\mu} \ (\mu\in Y^{\iota}),\quad  \text{and} \quad F_i,E_i \ (\alpha_i\in \Delta_{\bullet}).
    \end{equation*}
    Here $T_{w_{\bullet}}:=T_{w_{\bullet},+1}''$ in Lusztig's notation \cite{Lus4}.
\end{definition}

\begin{example}
    Let $\Delta$ be the set of simple roots for simple $\g$ and consider the corresponding split Satake diagram. Then $U_h^{\iota}$ is the subalgebra of $U_h(\g)$ generated by the elements $B_i=F_i+ \varsigma_iE_iK_i^{-1}+ \kappa_iK_i^{-1}$, for $\alpha_i\in \Delta$.
\end{example}

\begin{remark}
    Although it does not appear to be important in this paper, a fundamental aspect of the structure of the $\iota$quantum group is that it is a (left) coideal subalgebra of $U_{h}(\mathfrak{g})$
    \[
    \Delta: U^{\iota}_{h} \longrightarrow U^{\iota}_{h} \ot U_h(\mathfrak{g}),
    \]
    which specializes to $U(\mathfrak{g}^{\theta})$ modulo $h$, where $\theta$ is the involution associated to the Satake diagram as in \cite[Theorem 2.5]{MR3269184}.
\end{remark}

There is an $\iota$Serre presentation for quasi-split $\iota$quantum groups \cite{MR4309551}, which generalizes the following well-known presentation of the split $\iota$quantum group corresponding to the pair $\som\subset\slm$. 

\begin{theorem}\label{thm:iSerre}
    Let $\Delta$ be the set of simple roots for $\slm$ and consider the corresponding split Satake diagram. The $\iota$quantum group $U_h^{\iota}$ (with scalars $\varsigma_i=q_{i}^{-1}$ and $\kappa_i=0$, for all $\alpha_i\in \Delta$), is isomorphic to the algebra generated by $B_i=F_i+q_i^{-1}E_iK_i^{-1}$, $i=1, \dots, m-1$, modulo the $\iota$Serre relations
    \begin{equation}\label{eq:iserre}
    B_i^2B_{i\pm 1} - (q+q^{-1})B_iB_{i\pm 1}B_i +B_{i\pm 1}B_i^2 = B_{i\pm 1}
    \end{equation}
    and 
    \begin{equation}
    B_iB_j=B_jB_i \qquad |i-j|> 1.
    \end{equation}
\end{theorem}

\subsection{Integrable modules over $\iota$quantum groups}

Here we recall the relevant definitions from \cite{WW25}, see also \cite[Section 3.3]{MR4533490}.

\textbf{From here on, we will only be concerned with quasi-split $\iota$quantum groups with $\varsigma_i=q_{i}^{-1}$.}

    \begin{definition}
        For $n\in \mathbb{Z}$, let $p(n)\in \{\overline{0},\overline{1}\}$ be the parity. The \emph{$\iota$divided power} of $B_i$ is defined to be
        \[
        B_i^{(m)}:=\frac{B_i^m}{[m]_i!}, 
        \]
        for $\tau(i) \ne i$, while
        \begin{equation}\label{eq:idivpowereven}
        B_{i,\overline{0}}^{(m)}=\frac{1}{[m]_i!}B_i^{1+\delta_{p(m),\overline{0}}}\prod_{r=1}^{\lceil m/2 \rceil -1}(B_i^2-[2r]_i^2)
        \end{equation}
        and
        \begin{equation}\label{eq:idivpowerodd}
        B_{i,\overline{1}}^{(m)}=\frac{1}{[m]_i!}B_i^{\delta_{p(m),\overline{1}}} \prod_{r=1}^{\lfloor m/2 \rfloor}(B_i^2-[2r-1]_i^2),
        \end{equation}
        for $\tau(i) = i$.
    \end{definition}
    \begin{definition}\label{def:iquantumintegrable}
        A $U_h^{\iota}$-module $M$ is \emph{$X_{\iota}$-weighted} if $M$ has a decomposition
        \[
        M=\bigoplus_{\lambda\in X_\iota} M_{\lambda},
        \]
        such that
        \[
        K_{\mu}v=q^{\langle \mu, \lambda \rangle}v, \quad \text{for all $\mu\in Y^\iota$, $\lambda\in X_\iota$, and $v\in M_{\lambda}$}
        \]
        (recall that for $\mu \in Y^{\iota}$, we have $K_{\mu}\in U_h^{\iota}$), and such that
        \[
        B_i:M_{\lambda}\rightarrow M_{\lambda + \overline{\alpha_i}},\quad  \text{for all $\lambda \in X_{\iota}$ and $\alpha_i\in \Delta_{\circ}$}.
        \]
        An $X_{\iota}$-weighted $U_h^{\iota}$-module $M$ is called \emph{integrable} if, for all $\lambda\in X_\iota$, when $v\in M_{\lambda}$, then 
        \[
        B_j^{(m)}v=0 \ (\tau(j) \neq j) \quad \text{and} \quad B_{i,p(\langle \alpha_i^{\vee},\lambda \rangle)}^{(m)}v=0 \ (\tau(i) =i), \quad \text{for all $m\gg 0$},
        \]
        i.e., $M$ is locally nilpotent with respect to $\iota$divided powers of $B_i$.
    \end{definition}

\begin{remark}
    Parity $p(\langle \alpha_i^\vee, \lambda \rangle )$ is independent of choice of representative $\widetilde\lambda \in X$ of $\lambda\in X_\iota$.
\end{remark}

\begin{lemma}\label{lem:deformiquantummodule}
    For each finite-dimensional irreducible $\som$-module $M$ with integral highest weight there exists a unique (up to isomorphism) $U^{\iota}_h(\som \subset \slm)$-module structure on $M$, which specializes to the $\som$ action under the limit $h\rightarrow 0$. Moreover, the eigenvalues of the generators $B_i$ acting in such a representation are contained in $\{[\ell] \ | \ \ell \in \mathbb{Z}\}$.
\end{lemma}
\begin{proof}
    Follows from the explicit formulas for the action of the generators $B_i$ in the Gelfand-Tsetlin basis, and the classification of finite dimensional irreducible (classical) modules, found in \cite{MR2143754}.
\end{proof}

\begin{lemma}
    Let $\kompact = \som$. Let $M$ be a finite dimensional integrable $\kompact$-module (so the action is the differential of some finite dimensional $K=SO_m$-module). Then the $U_h^{\iota}(\som\subset \slm)$-module structure on $M$ is integrable as in Definition \ref{def:iquantumintegrable}.
\end{lemma}

\begin{proof}
    Since $M$ comes from an $SO_m$-module, it follows that $M$ is a direct sum of irreducible $\som$-modules, each  with integral highest weight. 
\end{proof}

    Note that in \cite{MR2143754}, it is proven that finite-dimensional $U_h^{\iota}(\som\subset \slm)$-modules are completely reducible. Thus, for $\som$-modules, we may say integrable to mean ``is the differential of a $K$-module", or equivalently, is integrable in the sense of Definition \ref{def:iquantumintegrable} when $h=0$.

\subsection{$\iota$quantum Weyl group} In this section we introduce the $\iota$quantum analogue of Lusztig's symmetries from Definition \ref{def:Lussymmetry}. These explicit series formulas in terms of $\iota$divided powers originally appeared in Weinan Zhang's thesis \cite{zhang-thesis}, and then were further developed in \cite{WW25}.
    
\textbf{Now, we will only be concerned with split $\iota$quantum groups with $\varsigma_i=q_i^{-1}$ and $\kappa_i=0$.} 

    \begin{definition}\label{def:iQW}
        Fix $\alpha_i\in \Delta_{\circ}$, such that $\tau(i) = i$. Let $M$ be an integrable $U_h^\iota$-module. For any $v\in M_{\lambda}$, $\lambda\in X_\iota$, we define
        \begin{equation}\label{iQ_symm}
            \iQW_i\cdot v:=\sum_{k\ge 0, \  p(k)=p(\langle \alpha_i^{\vee},\lambda \rangle)} (-q_i)^{-k/2} B_{i,p(k)}^{(k)}\cdot v
        \end{equation}
        and
        \begin{equation}\label{iQ_symm_inverse}
            \iQW_i^{-1}\cdot v:=\sum_{k\ge 0, \  p(k)=p(\langle \alpha_i^{\vee},\lambda \rangle)} (-q_i)^{k/2} B_{i,p(k)}^{(k)}\cdot v
        \end{equation}
        We call $\iQW_i$ the $i$-th generator of the \emph{$\iota$quantum Weyl group}, also known as the $i$-th \emph{relative braid symmetry} of $M$. 
    \end{definition}

The operators $\iQW_i$ and $\iQW_i^{-1}$ are denoted $\textbf{T}_{i,-1}'$ and $\textbf{T}_{i,+1}''$ respectively, in \cite{WW25}.

    \begin{theorem}[{\cite[Corollary 3.10]{WW25}}]
    The operator $\iQW_i$ is invertible with inverse $\iQW_i^{-1}$. 
    \end{theorem} 

Moreover, these operators satisfy braid relations.

    \begin{theorem}[{\cite[Theorem E]{WW25}}]
        Let $\alpha_i, \alpha_j\in \Delta_{\circ}$ such that $\tau(i) = i$ and $\tau(j) = j$. Then for an integrable $U_{h}^{\iota}$-module $M$ we have the following equality of endomorphisms of $M$:
        \[
        \iQW_i\iQW_j\iQW_i\cdots = \iQW_j\iQW_i\iQW_j\cdots,
        \]
        where both sides have $m_{ij}$ factors, with $m_{ij}$ being the order of $s_{i}s_{j}$ in the Weyl group of $\mathfrak{g}$.
    \end{theorem}

    \begin{corollary}\label{cor:iqwbraidrep}
        Let $M$ be an integrable $U_{h}^{\iota}$-module. There is a representation
        \[
        \pi_{bW, \g^{\theta}\subset \g}^{h}:{\rm Br}_{W}\rightarrow GL(M)
        \]
        where for $\alpha_i\in \Delta_{\circ}$, the $i$-th braid group generator maps to the invertible operator $\iQW_i$.
    \end{corollary}

    \begin{remark}
        Formulas for relative braid symmetries for quasi-split $\iota$quantum groups are developed in \cite{WW25}, but in Definition \ref{def:iQW} we only recall the formulas for split symmetric pairs. The definition of boundary Casimir connection in \cite{MR4983792} is only for split symmetric pairs, but see Question \ref{quest:quasisplit}.
    \end{remark}


\section{The dual pair $(U_{2h}(\son), U_{2h+\pi {\bf i}}^{\iota}(O_m))$}\label{sec5}


    To quantize classical $(\son, 
    O_m)$ Howe duality, we will work with $(U_{2h}(\son[2n]), U^{\iota}_{2h+\pi {\bf i}}(\mathfrak{so}_m))$, where
        \begin{itemize}
            \item $U_{2h}(\son[2n])$ is a quantum group, with $\Delta=D_n$, as in Definition \ref{Quantum_group} (but with $h$ replaced by $2h$), and
            \item $U^{\iota}_{2h+\pi {\bf i}}(\mathfrak{so}_m)$ is an $\iota$quantum group, with $\Delta_{\circ}=A_{m-1}$, $\Delta_{\bullet} = \emptyset$, $\tau=\id$, $\varsigma_i=q_i^{-1}$, and $\kappa_i=0$, as in Definition \ref{iQuantumgroup}.
        \end{itemize}

\begin{notation}
    Our convention is that $q:=e^h$, and, since we work with $d_i=2$, that 
    \begin{itemize}
        \item $q_i=e^{2h}=q^2$, when considering $U_{2h}(\son)$, and
        \item $q_i=e^{2h+\pi {\bf i}} = -q^2$, when considering $U_{2h+\pi {\bf i}}^{\iota}(\som)\subset U_{2h+\pi {\bf i}}^{\iota}(\slm)$. 
    \end{itemize}
    If it is ambiguous, then we will specify $q_i=q^2$ or $q_i=-q^2$. When convenient, we may write $Q:=-q^2$.
\end{notation}

\subsection{Spin representation for $U_{2h}(\son[2n])$}

Let $S$ be the spin representation for $\son$ as in \S\ref{spinrepforson}. We will define an action of $U_{2h}(\son[2n])$ on $S$. To this end, we introduce the following operators in $\End(S)$
\begin{equation*}
    \omega_i:=\psi_i \psi_i^{*}+q^{-1}\psi_i^{*}\psi_i, \quad i=1, \dots, n.
\end{equation*}
For any $I\subset \{1,\dots,n\}$, with $I=\{i_1< i_2 < \dots < i_{|I|}\}$, there is a vector $\hwp_{I}:= \psi_{i_1}^{*} \dots \psi_{i_{|I|}}^{*}\in S$, as in Notation \ref{not:orderedbasisforS}. Note that, 
\[
\omega_i\cdot \hwp_{I} = \begin{cases}
    \hwp_I \quad \text{if $i\notin I$, and }\\
    q^{-1}\hwp_I \quad \text{if $i\in I$},
\end{cases}
\]
in particular the operators $\omega_i$ are invertible and commute. 

\begin{proposition}[{\cite[Proposition 3.1]{Wenzl-Spin}}]\label{QG_so2n_action}
    There is an algebra homomorphism 
    \[
    \rho:U_{2h}(\son)\rightarrow \End(S)
    \]
    given by
\begin{equation*}
    \rho(E_i)=\psi_i\psi_{i+1}^{*}\omega_i\omega_{i+1}^{-1}, \qquad \rho(F_i)=(\omega_i\omega_{i+1}^{-1})^{-1}\psi_{i+1}\psi_{i}^{*},
\end{equation*}
\begin{equation*}
    \rho(E_n)=\psi_{n-1}\psi_n q\omega_{n-1}\omega_{n}, \qquad \rho(F_n)=(q\omega_{n-1}\omega_{n})^{-1}\psi_n^{*}\psi_{n-1}^{*}
\end{equation*}
\begin{equation*}
    \rho(K_i)= \omega_i^2 \omega_{i+1}^{-2}, \qquad \rho(K_n)= q^2 \omega_{n-1}^2 \omega_n^{2}.
\end{equation*}
\end{proposition} 

\begin{remark}
    We adjusted Wenzl's formulas to match our conventions.
\end{remark}

The action of $U_{2h}(\son)$ on $S$ gives rise to an action on $S^{\otimes m} $, using the coproduct $\Delta$ from Definition \ref{Quantum_group}. Consider the case when $m=2$ and recall the decomposition of $S\ot S$ into the four summands
\begin{equation}\label{eq:foursummands}
S^{+}\ot S^{+}, \quad S^{-}\ot S^{-}, \quad S^{+}\ot S^{-}, \quad \text{and} \quad S^{-}\ot S^{+},
\end{equation}
using the notation of \eqref{eq:Spmisohwt}. 

\begin{lemma}
    Each of the four summands in \eqref{eq:foursummands} is preserved by the action of $U_{2h}(\son)$.
\end{lemma}
\begin{proof}
The action of $U_{2h}(\son)$ on $S$ is given by even operators. 
\end{proof}

For each of the four summands in \eqref{eq:foursummands} we will now explicitly describe all $U_{2h}(\son)$-singular vectors (up to scalar).

\begin{proposition}\label{high_vect}
    Let $\epsilon_1, \epsilon_2 \in \{\pm\}$. The $k$-th highest weight vector $\hwp_k^{\epsilon_1\epsilon_2}\in S^{\epsilon_1}\otimes S^{\epsilon_2}$, is given by
\begin{equation}\label{eq:highvect++}
    \hwp_{k}^{++}=\sum_{J\subset I_{2k},2\mid\#J} c^{+}_J\cdot \hwp_{J}\otimes \hwp_{I_{2k}\setminus J}, \qquad \hwp_0^{++}=1\otimes 1,
\end{equation}
\begin{equation}
    \hwp^{--}_k = \sum_{J\subset I_{2k},2\nmid\#J} c^{-}_J\cdot \hwp_{J}\otimes \hwp_{I_{2k}\setminus J}, \qquad  \hwp_0^{--}=\psi_n^{*}\otimes \psi_n^{*}, 
\end{equation}
\begin{equation}
    \hwp^{+-}_k = \sum_{J\subset I_{2k+1},2\mid\#J} c^{+}_J\cdot \hwp_{J}\otimes \hwp_{I_{2k+1}\setminus J}, \qquad  \hwp_0^{+-}=1\otimes \psi_n^{*}, \qquad \text{and}
\end{equation}
\begin{equation}
\hwp^{-+}_k = \sum_{J\subset I_{2k+1},2\nmid\#J} c^{-}_J\cdot \hwp_{J}\otimes \hwp_{I_{2k+1}\setminus J}, \qquad \quad \hwp_0^{-+}=\psi_n^{*}\otimes 1.
\end{equation}
Here, we use the notation
\begin{itemize}
    \item $\hwp_I=\psi_{i_1}^*\dots \psi_{i_{|I|}}^*$, $I=\{i_1 < \dots < i_{|I|}\}$, see Notation \ref{not:orderedbasisforS},
    \item $I_{k}:=\{n-k+1,\dots,n\}$, so $\hwp_{I_k}=\psi_{n-k+1}^{*}\dots\psi_{n}^{*}$ (letting $I_0:=\emptyset$ and $\hwp_{I_0} = 1$),
    \item $c^{\epsilon}_J:=(-q^2)^{\hgt(\lambda_{low, \epsilon} -\lambda_J)}$, for $\epsilon\in \{\pm 1\}$,
    \item $\lambda_J = \wt (\hwp_{J})$ is the weight of the corresponding monomial,
    \item $\lambda_{low,\pm}$ is the lowest weight of the corresponding module $S^{\pm}$, and
    \item $\hgt(\lambda):=\sum_{i=1}^{n}m_i$, for a weight $\lambda=\sum_{i=1}^{n} m_i \alpha_i$.
\end{itemize}
These are all the singular vectors (up to scalar) in $S^{\epsilon_1}\otimes S^{\epsilon_2}$.
\end{proposition}

\begin{proof}
    It is a straightforward computation to check these are singular vectors of the specified weight. That they exhaust all singular vectors (up to scalar) follows from Proposition \ref{FH_decomposition} and Theorem \ref{thm:equivalencequantumclassical}.
\end{proof}

\subsection{Commuting action of $U_{2h+\pi {\bf i}}^{\iota}(\som)$ on $\howespin$}

The $\iota$quantum group $U^{\iota}_{2h+\pi {\bf i}}(\mathfrak{so}_m)$ is the subalgebra of $U_{2h+\pi {\bf i}}(\slm)$ generated by 
    \begin{equation}\label{eq:quantumBigenerators}
        B_i=F_i+q^{-1}_iE_iK_i^{-1}=F_i-q^{-2}E_iK_i^{-1}, \quad i=1, \dots, m-1.
    \end{equation}
    These elements satisfy the $\iota$Serre relations \eqref{eq:iserre}:
    \begin{equation}
    B_i^2B_{i\pm 1} + (q^2+q^{-2})B_iB_{i\pm 1}B_i +B_{i\pm 1}B_i^2 = B_{i\pm 1}
    \end{equation}
    and 
    \begin{equation}
    B_iB_j=B_jB_i \qquad |i-j|> 1.
    \end{equation}
    By Theorem \ref{thm:iSerre}, the generators $B_i$, and the $\iota$Serre relations, give generators-and-relations for $U^{\iota}_{2h+\pi {\bf i}}(\mathfrak{so}_m)$.

We will construct an action of this $\iota$quantum group using the following elements in $\End(S^{\otimes 2})$.
\begin{definition}
Let
\begin{equation*}
    C:=\sum_{i=1}^n \Omega_{i-1}^{-1}\otimes\Omega_{i-1}(\psi_i \otimes \psi_i^{*}+\psi_i^{*}\otimes \psi_i),
\end{equation*}
where
\begin{equation*}
    \Omega_i :=\prod_{j=1}^{i}\omega_j^2, \quad 1\le i\le n, \quad \text{and} \quad \Omega_0:=1.
\end{equation*}
\end{definition}

Consider the $U_{2h}(\son)$-module $\howespin\cong S^{\ot m}$, $\hwp_I\mapsto \hwp_{I_1}\otimes \dots \hwp_{I_m}$, as in Notation \ref{not:psibasis}. 

\begin{definition}
    Let 
    \begin{equation}\label{c_i}
        C_i:=\id_{S^{\ot i-1}}\ot C \ot \id_{S^{\ot m-i-1}}.
    \end{equation}
\end{definition}

\begin{proposition}[{\cite[Proposition 4.2]{Wenzl-Spin}}]\label{b-action}
    There is an algebra homomorphism
    \begin{align*}
        U^{\iota}_{2h+\pi {\bf i}}(\mathfrak{so}_m)&\rightarrow \End_{U_{2h}(\son)}(S^{\otimes m}). \\
        B_i&\mapsto C_i
    \end{align*}
    In particular, the $C_i$ commute with the action of $U_{2h}(\son)$ on $S^{\otimes m}$ determined by Proposition \ref{QG_so2n_action}.
\end{proposition}

\begin{remark}
    The proof of Proposition \ref{b-action} uses a generators and relations presentation of $U^{\iota}_{2h+\pi {\bf i}}(\mathfrak{so}_m)$. In \cite[Proposition 4.2]{Wenzl-Spin}, Wenzl checks that the $C_i$ satisfy the $\iota$Serre relations \ref{eq:iserre}. That the $C_i$ satisfy the $\iota$Serre relations is given a diagrammatic interpretation in \cite{MR4824565, mcnamara2025quantumspinbrauercategory}. A different description of the action of the $\iota$quantum group on $\howespin$ is by restricting the action of quantum $\slm$ on $\howespin$ from quantum skew Howe duality \cite{aboumrad2022skewhowedualitytypes}.
\end{remark}

We now compute the action of $C$ on highest weight vectors in $S^{\ot 2}$.

\begin{proposition}\label{C-action}
    Let $Q=-q^2$. Then we have
    \begin{equation}\label{eq:2kand2kp1}
        C\hwp^{\pm \pm}_k= [2k]_Q \hwp_k^{\mp\mp} \quad \text{and} \quad C\hwp^{\pm \mp}_k= [2k+1]_Q \hwp_k^{\mp\pm}.
    \end{equation}
\end{proposition}

\begin{proof}
    See Appendix \ref{proof5.12}.
\end{proof}

Before we proceed to work with the braid group symmetries, note that $X_\iota =\oplus_{i=1}^{m-1} \mathbb{Z}/2\mathbb{Z}\cdot \overline{\varpi_i}$. Let us define an $X_\iota$-grading on $\howespin$ so that the submodule
\begin{equation}\label{eq:igradingeven}
    S^{\ot i-1}\ot \big((S^+\ot S^+) \oplus (S^- \ot S^-)\big)\ot S^{\ot m-i-1}, \quad 1\le i \le m-1,
\end{equation}
is even, with respect to $\alpha_i^{\vee}\in \Delta^{\vee}$, and the submodule
\begin{equation}\label{eq:igradingodd}
    S^{\ot i-1}\ot \Big((S^+\ot S^-) \oplus (S^- \ot S^+)\Big)\ot S^{\ot m-i-1}, \quad 1\le i \le m-1,
\end{equation}
is odd with respect to $\alpha_i^{\vee}\in \Delta^{\vee}$. 

\begin{example}
    If $m=3$, so $X_{\iota}\cong \mathbb{Z}/2\times \mathbb{Z}/2$, then $\howespin_{(\overline{0},\overline{0})} = (S^+\otimes S^+\otimes S^+) \oplus (S^-\otimes S^-\otimes S^-)$ and $\howespin_{(\overline{0},\overline{1})} = (S^+\otimes S^+\otimes S^-)\oplus (S^-\otimes S^-\otimes S^+)$.
\end{example}

Thus, we have a decomposition
\begin{equation}\label{eq:SXiotagrading}
    \howespin = \bigoplus_{\lambda\in X_{\iota}}\howespin_{\lambda}
\end{equation}
with respect to the specified $X_{\iota}$-grading.

\begin{lemma}
    Each summand in the decomposition in \eqref{eq:SXiotagrading} is preserved by the action of $U_{2h}(\son)$. The decomposition in \eqref{eq:SXiotagrading} makes $\howespin$ an $X_{\iota}$-weighted module for $U_{2h+\pi {\bf i}}^\iota (\som)$, as in Definition \ref{def:iquantumintegrable}.
\end{lemma} 
\begin{proof}
    Follows by parity considerations, using the explicit descriptions of the action of generators in Proposition \ref{QG_so2n_action} and Proposition \ref{b-action}
\end{proof}

\begin{lemma}
    The $X_{\iota}$-weighted module $\howespin$ is an integrable $U_{2h+\pi {\bf i}}^\iota (\som)$-module, as in Definition \ref{def:iquantumintegrable}.
\end{lemma}  
\begin{proof}
    Since $\howespin\cong S^{\otimes m}$ and each $C_i$ acts on the tensor components $i$ and $i+1$, Proposition \ref{C-action} shows that the eigenvalues of $C_i$ on $\howespin$ are $\pm [\ell]_Q$, for $0\le \ell\le n$. The singular vectors in Proposition \ref{C-action} are $X_{\iota}$-graded, so combining \eqref{eq:2kand2kp1} with \eqref{eq:igradingeven} and \eqref{eq:igradingodd}, we find that, for $\epsilon \in \{0,1\}$,
    \[
    \bigoplus_{\ell \in \epsilon + 2\mathbb{Z}}\{v\in \howespin \ | \ C_iv=[\ell]_{Q}v\} = \bigoplus_{\substack{\lambda \in X_{\iota} \\ p(\langle\alpha_i^{\vee}, \lambda\rangle)= \epsilon}}\howespin_{\lambda}.
    \]
    Since the action of $U_{2h+\pi {\bf i}}^\iota (\som)$ on $\howespin$ is by $B_i\mapsto C_i$, it follows that $\howespin$ is locally nilpotent with respect to the $\iota$divided powers, in the sense of Definition \ref{def:iquantumintegrable}. Thus, $\howespin$ is integrable. 
\end{proof}

\begin{lemma}\label{grading}
    Let $m\ge 2$, and let $M$ be a simple finite-dimensional $U^\iota_{2h+\pi{\bf i}}(\som)$-module. If $M$ admits an $X_{\iota}$-grading for which $M$ is integrable with respect to it, then such a grading is unique.
\end{lemma}
\begin{proof}
    If $v\in M=\oplus_{\lambda\in X_{\iota}}M_{\lambda}$ is a non-zero graded vector, then $v\in M_{\lambda}$ if and only if
    \[
    B_{i,p(\langle \alpha_i^{\vee},\lambda \rangle)}^{(m)}v=0, \quad \text{for all $m\gg 0$ and for $i=1, \dots, m-1$.}
    \]
    The forward implication is just Definition \ref{def:iquantumintegrable}, while the reverse implication follows by using that $\iota$divided powers of the form $B_{i,\overline{0}}^{(2k+1)}$ and $B_{i,\overline{1}}^{(2k)}$, viewed as polynomials in $B_i$, each are a product of distinct linear factors and do not have a factor in common. Since $M$ is irreducible and $0\ne v\in M$, $M$ is cyclic and generated by $v$. Thus, 
    \[
    M = U_{2h+\pi\imag}^{\iota}(\som)\cdot v = \mathrm{span}\{B_{i_1}B_{i_2}\dots B_{i_k}\cdot v \ | \ i_1, i_2, \dots, i_k\in \{1, \dots, m-1\}\}
    \]
    For $X_{\iota}$-weighted modules, we have\footnote{This condition is part of Definition \ref{def:iquantumintegrable}, but it is known to experts that the relations \eqref{eq:iserre} imply this condition for integrable modules.} $B_i:M_{\lambda}\rightarrow M_{\lambda + \overline{\alpha_i}}$, so each vector $B_{i_1}B_{i_2}\dots B_{i_k}\cdot v\in M$ is $X_{\iota}$-graded. The statement follows.
\end{proof}

\subsection{Double centralizer by commuting action of $U_{2h+\pi \rm{i}}(O_m)$ on $\howespin$}
The action of $U^{\iota}_{2h+\pi {\bf i}}(\mathfrak{so}_m)$ does not generate $\End_{U_{2h}(\son)}(S^{\otimes m})$. This parallels the classical orthogonal Howe duality, which is for $(O_m, \son)$, not $(\mathfrak{so}_m, \son)$. To generate $\End_{U_{2h}(\son)}(S^{\otimes m})$, we introduce the following algebra, which is a quantum analogue of $O_m$. 

\begin{definition}
    Define $ U^{\iota}_{2h+\pi {\bf i}}(O_m):=(\mathbb{Z}/2\mathbb{Z})\ltimes  U^{\iota}_{2h+\pi {\bf i}}(\mathfrak{so}_m)$ where the generator of $\mathbb{Z}/2\mathbb{Z}$ acts by $-1$ on the generator $B_1\in U^{\iota}_{2h+\pi {\bf i}}(\mathfrak{so}_m)$ and by $1$ on all other generators $B_2, \dots, B_{m-1}\in U^{\iota}_{2h+\pi {\bf i}}(\mathfrak{so}_m)$.
\end{definition}

\begin{lemma}
    There is an action of $ U^{\iota}_{2h+\pi {\bf i}}(O_m)$ on $\howespin\cong S^{\otimes m}$, extending the $U^{\iota}_{2h+\pi {\bf i}}(\mathfrak{so}_m)$ action in Proposition \ref{b-action}, such that the generator of $\mathbb{Z}/2\mathbb{Z}$ acts by $e^{\pi \rm{i} E_{1,1}}$.
\end{lemma}
\begin{proof}
    The operator $C_1$ acts on the first two tensor factors by odd operators in each tensor component.
\end{proof}

There is a quantum analogue of Theorem \ref{Orthogonal_Howe_duality}.

\begin{theorem}[{\cite[Theorem 5.3(b)]{Wenzl-Spin}}]\label{thm:quantumorthogonalHowe}
      The actions of $U_{2h+\pi {\bf i}}^{\iota}(O_m)$ and $U_{2h}(\son)$ commute on $\howespin$ and centralize each other. Under the action of the pair $(U_{2h}(\son), U^{\iota}_{2h+\pi {\bf i}}(O_m))$, $\howespin$ decomposes as
      \begin{equation}
          \howespin\cong \bigoplus_{\substack{\lambda=(\lambda_1,\dots, \lambda_m) \\ \lambda^{\prime}_1+\lambda^{\prime}_2\le m, \ l(\lambda^{\prime})\le n}} V^{\son}_{2h,\lambda^{\top}}\ot V^{O_m}_{2h+\pi{\bf i},\lambda}, 
      \end{equation}
      where
\begin{itemize}
    \item the combinatorial notations $\lambda$, $\lambda'$, $l(-)$, and $\lambda^{\top}$, are as in Theorem \ref{Orthogonal_Howe_duality},
    \item $V^{\son}_{2h,\lambda^{\top}}$ is the (unique) irreducible $U_{2h}(\son)$-representation of highest weight $\lambda^\top$ (see Theorem \ref{thm:equivalencequantumclassical}), and
    \item $V^{O_m}_{2h+\pi{\bf i},\lambda}:= \Hom_{U_{2h}(\son)}(V^{\son}_{2h,\lambda^{\top}}, \howespin) $ is the (unique) irreducible $U^{\iota}_{2h+\pi{\bf i}}(O_m)$-representation labelled by $\lambda$.
\end{itemize}
\end{theorem}

\subsection{Commuting action of $U_{2h}^{\iota}(\som)$ on $\howespin$}
Although the action of $U^{\iota}_{2h+\pi {\bf i}}(\mathfrak{so}_m)$ is important for setting up the quantum orthogonal Howe duality, we will also need to utilize a related action of $U^{\iota}_{2h}(\mathfrak{so}_m)$ on $S^{\otimes m}$.
\begin{definition}
Define
\[
\widetilde{C}_i:= \id_{S^{\ot i-1}}\ot \widetilde{C} \ot \id_{S^{\ot m-i-1}},
\]
where
\begin{equation}\label{tilde_c}
    \widetilde{C}={\bf i}\sum_{i=1}^n (\Omega_{i-1}^{-1}\ot \Omega_{i-1})\cdot (\psi_{i1}^*\psi_{i2}-\psi_{i2}^* \psi_{i1}).
\end{equation}
\end{definition}
\begin{lemma}\label{lem:tildeCisC}
We have the equality
\[
\widetilde{C}_i= {\bf i} e^{\pi {\bf i} E_{i,i}} C_i
\]
of operators on $\howespin\cong S^{\otimes m}$.
\end{lemma}
\begin{proof}
    Straightforward computation.
\end{proof}

\begin{proposition}\label{sign_switch}
    We have an algebra homomorphism 
    \[
    U^{\iota}_{2h}(\som \subset \slm) \rightarrow \End_{U_{2h}(\son)}(S^{\ot m}),
    \]
    such that
    \[
    \qquad B_i \mapsto \widetilde{C}_i.
    \]
    Furthermore, when $h=0$, so $U^{\iota}_{0}(\som \subset \slm) \cong U(\som)$, this algebra homomorphism specializes to a classical $\som$ action given by
    \[
    B_i \mapsto {\bf i}(E_{i,i+1}-E_{i+1,i}).
    \]
\end{proposition}
\begin{proof}
    We know that the elements $C_i$ satisfy the $\iota$Serre relations for $U^{\iota}_{2h+\pi {\bf i}}(\som \subset \slm)$, that is
    \[
    C_{i}^2 C_{i\pm 1}+(q^2+q^{-2})C_i C_{i\pm 1} C_{i}+C_{i\pm 1}C_i^2=C_{i\pm 1},
    \]
    and
    \[
    C_iC_j=C_jC_i, \quad |i-j|>1.
    \]
    By utilizing Lemma \ref{lem:tildeCisC}, along with the fact that $e^{\pi {\bf i} E_{j,j}} C_{i}=-C_{i}e^{\pi {\bf i} E_{j,j}}$, when $j=i$ or $i+1$, and $e^{\pi {\bf i} E_{j,j}}$ commutes with $C_i$ otherwise, we get the relations
    \[
    \widetilde{C}_i^{2}\widetilde{C}_{i\pm 1}-(q^2+q^{-2})\widetilde{C}_i \widetilde{C}_{i\pm 1} \widetilde{C}_i + \widetilde{C}_{i\pm 1}\widetilde{C}_i^2= \widetilde{C}_{i\pm 1}
    \]
    and
    \[
    \widetilde{C}_i\widetilde{C}_j=\widetilde{C}_j\widetilde{C}_i, \quad |i-j|>1.
    \]
    Thus, using the presentation of $U_{2h}^{\iota}(\som\subset \slm)$ in Theorem \ref{thm:iSerre}, we get the desired homomorphism. 
    
    For the second statement we observe that when $h=0$ the elements $\Omega_i$ act by $1$ on each respective $S$ component. Therefore, according to \eqref{eq:Babactionsom}, $\widetilde{C}_i$ act by ${\bf i} (E_{i,i+1}-E_{i+1,i})\in \som$.
\end{proof}

%
\section{Braid group actions on tensor powers of type $D$ spin representation}\label{sec6}
%

 In this section we will follow the arguments from \cite[Section 4]{BER}, in which an analogue of Theorem \ref{R-matrix-to-i-q-Weyl} is proven for the type $B$ spin representation. We will use the notation $Q=-q^2$.

\subsection{$R$-matrix action on singular vectors}
We describe the action of $R_{S,S}$ on the highest weight vectors in $S^{\ot 2}$, see Proposition \ref{high_vect} and \eqref{R-matrix}.

\begin{proposition}\label{R-action}
    We have
\begin{equation}
    R_{S,S}(\hwp_k^{\pm \pm})=(-1)^kq^{\frac{n-2(2k)^2}{2}} \hwp_k^{\pm \pm} \qquad \text{and} \qquad R_{S,S}(\hwp_k^{\pm \mp})=(-1)^kq^{\frac{n-2(2k+1)^2}{2}} \hwp_k^{\mp \pm}.
\end{equation}
\end{proposition}
\begin{proof}
See Appendix \ref{proof6.1}.
\end{proof}

\subsection{$\iota$quantum Weyl group action on singular vectors}

In Proposition \ref{C-action}, we computed the action of $C\in \End(S^{\otimes 2})$ on the singular vectors. Since the $\iota$quantum Weyl group is defined in terms of the $\iota$divided powers of the generators $B_i$, which map to $C_i=\id\otimes C\otimes \id\in \End(S^{\otimes m})$, we are reduced to computing how the $\iota$quantum Weyl group generators act on $C$-eigenspaces. 

To this end, let us consider polynomials
\begin{equation}\label{eq:evendivpowerpoly}
    B_{\overline{0}}^{(2k)}(b):=\frac{b^{2}(b^2-[2]_Q^2)(b^2-[4]_Q^2)\dots (b^2-[2k-2]^2_Q)}{[2k]_Q!}
\end{equation}
and
\begin{equation}\label{eq:odddivpowerpoly}
    B_{\overline{1}}^{(2k+1)}(b):=\frac{b(b^2-[1]_Q^2)(b^2-[3]_Q^2)\dots (b^2-[2k-1]^2_Q)}{[2k+1]_Q!},
\end{equation}
corresponding to even $\iota$divided powers \eqref{eq:idivpowereven} and odd $\iota$divided powers \eqref{eq:idivpowerodd}, as well as series
\begin{equation}\label{eq:evenseries}
    \iQW^{-1}_{\overline{0}}(b):=\sum_{k\ge 0} (-Q)^kB_{\overline{0}}^{(2k)}(b)
\end{equation}
and
\begin{equation}\label{eq:oddseries}
    (-Q)^{-1/2}\cdot \iQW_{\overline{1}}^{-1}(b):=\sum_{k\ge 0} (-Q)^k B_{\overline{1}}^{(2k+1)}(b),
\end{equation}
corresponding to the (inverses of the) respective $\iota$quantum Weyl group generators \eqref{iQ_symm}.

\begin{proposition}\label{i-q_gens}
    The following identities hold
    \begin{equation}
    \iQW_{\overline{0}}^{-1}([2k]_Q)=(-1)^k Q^{2k^2} \qquad \text{and} \qquad \iQW_{\overline{1}}^{-1}([2k+1]_Q)=(-Q)^{1/2} (-1)^kQ^{2(k^2+k)}.
\end{equation}
\end{proposition}
\begin{proof}
    See Appendix \ref{proof6.2}.
\end{proof}

\subsection{Identification of $R$-matrix and $\iota$quantum Weyl group}

Consider the setting of \S\ref{sec5}, with the element $B_i\in U_{2h+\pi{\bf i}}^\iota (\mathfrak{so}_m)$ acting via $C_i$ on $S^{\ot m}$. We wish to compute the action of $\iQW_i$ on $S^{\otimes m}$, where $\iQW_i$ is the $\iota$quantum Weyl group generator associated to $B_i$ defined by the series \eqref{iQ_symm}. 

For this computation, it suffices to restrict to the case $m=2$ and $i=1$, so $B=B_1\in U_{2h+\pi{\bf i}}^\iota (\mathfrak{so}_2)$ acting on $S^{\otimes 2}$. Now, it is enough to compute how $\iQW_1=\iQW$ acts on the highest weight vectors $\hwp_k^{\pm\pm}$ and $\hwp_k^{\pm\mp}$ in $S^{\otimes 2}$, from Proposition \ref{high_vect}. We already computed the action of $C_1=C$ on these highest weight vectors in Proposition \ref{C-action}. Since the action of $U_{2h+\pi{\bf i}}^\iota (\mathfrak{so}_2)$ on $S^{\otimes 2}$ is by $B_1\mapsto C_1$, it is then elementary to deduce how the operator $\iQW$ acts.

\begin{theorem}\label{R-matrix-to-i-q-Weyl}
The $R$-matrix \eqref{R-matrix} for $U_{2h}(\son)$ and the $\iota$quantum Weyl group generator $\iQW$ \eqref{iQ_symm} for $U_{2h+\pi{\bf i}}^\iota (\mathfrak{so}_2)$ are related by
\begin{equation}\label{R_iQW_comp}
    R_{S,S}=q^{\frac{n}{2}} \iQW
\end{equation}
as endomorphisms of $S^{\ot 2}$.
\end{theorem}

\begin{proof}
Since $C\in \End_{U_{2h}(\son)}(S^{\otimes 2})$, the action of $B\in U^{\iota}_{2h+\pi \imag}(\mathfrak{so}_2)$ commutes with the action of $U_{2h}(\son)$ on $S^{\ot 2}$, and it follows that the action of $\iQW$ also commutes with the action of $U_{2h}(\son)$. Note that $S^{\otimes 2}$ is generated over $U_{2h}(\son)$ by a spanning set of highest weight vectors (by virtue of being finite-dimensional). Since the vectors $\hwp_k^{\pm\pm},\hwp_k^{\pm\mp}$ span the subspace of highest weight vectors in $S^{\ot 2}$, it suffices to prove the equality \eqref{R_iQW_comp} holds on the vectors $\hwp_k^{\pm\pm},\hwp_k^{\pm\mp}$.

Combining Propositions \ref{C-action} and \ref{i-q_gens}, then comparing to Proposition \ref{R-action}, it follows that:
    \begin{equation}\label{R_iQW_hr_comp}
    R_{S,S}\hwp_k^{\pm \pm}= q^{\frac{n}{2}}\left(\iQW^{-1}_{\overline{0}} ([2k]_{Q})\right)^{-1}\hwp_k^{\pm \pm} \quad \text{and} \quad  R_{S,S}\hwp_k^{\pm \mp}= q^{\frac{n}{2}} \left(\iQW^{-1}_{\overline{1}}([2k+1]_{Q})\right)^{-1}\hwp_k^{\mp \pm}.
\end{equation}
Since $\iQW^{-1}_{\overline{0}}$ and $\iQW^{-1}_{\overline{1}}$ are the coefficients of the respective restrictions of the symmetry $\iQW^{-1}$ to the even and odd parity subspaces of $S^{\ot 2}$ with respect to the coroot $\alpha_1^{\vee}$, the desired equality \eqref{R_iQW_comp} holds on the vectors $\hwp_k^{\pm\pm},\hwp_k^{\pm\mp}$.
\end{proof}

\subsection{$\iota$quantum Weyl group for $U_{2h}^{\iota}(\som\subset \slm)$}\label{subsec:iquantumWeyl}

We would like to re-express the action of $\iQW$ on $S^{\otimes 2}$ in terms of $U_{2h}^{\iota}(\som[2]\subset \mathfrak{sl}_2)$, as opposed to $U_{2h+\pi {\bf i}}(\mathfrak{so}_2 \subset \mathfrak{sl}_2)$.

\begin{lemma}\label{iQW_pm_q2}
    The $\iota$quantum Weyl group symmetries for $C_i$ and $\widetilde{C}_i$ on $S^{\ot m}$ are related as follows
    \[
        e^{\pi {\rm i}(E_{i,i}+E_{i,i}E_{i+1,i+1})}\iQW_i(C_i)= e^{\pi {\bf i}E_{i,i}}(i,i+1)\flip^{-1}_{i,i+1} \iQW_i(C_i)=\iQW_i(\widetilde{C}_i). 
    \]
\end{lemma}

\begin{proof}
See Appendix \ref{prooflem6.4}.
\end{proof}

%
\section{Monodromic realisation of $\iota$quantum Weyl group operators}\label{sec7}
%

We now prove Theorem \ref{thm:main}, by establishing the following.

\begin{theorem}\label{monodromy}
    Let $(\mathfrak{k}\subset \mathfrak{g}) = (\mathfrak{so}_m\subset \mathfrak{sl}_m)$ and let $M$ be a finite-dimensional irreducible integrable $\mathfrak{k}$-module. Let 
    \[
    \pi_{bCas,\kompact\subset \g}^h:\BrW\rightarrow GL(M)
    \]
    be the monodromy representation of the connection $\nabla_{bCas, \kompact\subset \g}^{h,M}=d- h\sum_{\alpha\in \Phi_+}\frac{d\alpha}{\alpha}B_{\alpha}^2|_M$, for $h\in\C$, as in Proposition \ref{parralel_transport}. Let 
    \[
    \pi_{bW, \kompact\subset \g}^{2h}:\BrW\rightarrow GL(M)
    \]
    be the representation obtained by treating $M$ as a $U^{\iota}_{2h}(\mathfrak{k}\subset \mathfrak{g})$-module, as in Lemma \ref{lem:deformiquantummodule}, and acting by generators $\iQW_i$ of the $\iota$quantum Weyl group, as in Corollary \ref{cor:iqwbraidrep}. Then for generic $h$ (and for $h=0$) we have
    \[
    \pi^h_{bCas,\kompact\subset \g} \cong \pi_{bW, \kompact \subset \g}^{2\pi {\bf i}h},
    \]
    an isomorphism of representations of $\Br_W$.
\end{theorem}

\begin{proof}
   Consider the boundary Casimir connection with the fiber given by the $\som$-module $\mathbb{S}$. On the same space we have the KZ connection given by \eqref{KZ_so2n}. Due to Proposition \ref{Gauge_equiv} and Corollary \ref{Br_S_m_monodromy} we can relate the monodromy representations $\Br_W\rightarrow GL(\howespin)$ of $\nabla^{h,\howespin}_{bCas, \som \subset \slm}$ and $\nabla^{2h,\howespin}_{KZ,\son}$ as follows:
    \begin{equation}
        \pi^h_{bCas, \som \subset \slm}(\sigma_a)=e^{\pi {\rm i}(E_{a,a}+E_{a,a}E_{a+1,a+1})}e^{-\frac{\pi {\rm i}nh}{2}}\pi^{2h}_{KZ,\son}(\sigma_a),
    \end{equation}
    where $\sigma_a$ is the $a$-th braid group generator. 

    According to the Kohno--Drinfeld Theorem \cite{Drinfeld2,Koh1}, the monodromy representation and $R$-matrix representation for connections with fiber $\howespin\cong S^{\otimes m}$ are isomorphic in the context of Theorem \ref{thm:equivalencequantumclassical}
    \[
    \pi^{2h}_{KZ,\son} \cong \pi^{2\pi {\bf i}h}_{R, \son}:\Br_{\mathrm{S_m}}\rightarrow GL(\howespin).
    \]
    Namely, we have
    \begin{equation}
        F^{(m)}_{2h}\pi^{2h}_{KZ,\son}(\sigma_a) (F^{(m)}_{2h})^{-1} = \pi^{2\pi {\bf i}h}_{R, \son} (\sigma_a),
    \end{equation}
    where $F^{(m)}_{2h}$ is the Drinfeld twist from the Remark \ref{Artin's_lifting}. Note that $F^{(m)}_{2h}$ commutes with the $(\mathbb{Z}/2\mathbb{Z})^{\oplus m}$-grading of $\mathbb{S}$ given by $e^{\pi {\rm i}E_{aa}}$.

        Fix a partition $\lambda = (\lambda_1, \dots, \lambda_m)$ such that $\lambda_1'+\lambda_2'\le m$ and an integer $n\ge l(\lambda')$. Consider the representation map $U(\son)\rightarrow \End(\mathbb{S})$. Due to the orthogonal Howe duality as in Theorem \ref{Orthogonal_Howe_duality} and the density Theorem, we may choose an isotypic idempotent $e_0\in \End(\mathbb{S})$ corresponding to the module $V^{\son}(\lambda^{\top})$. Furthermore, since $e_0$ commutes with the $\som$-action on $\mathbb{S}$ we may restrict the fiber of $\nabla_{bCas,\som\subset \slm}^{h, \howespin}$ to the image of $e_0$. The monodromy representation of the boundary Casimir connection restricts to the image of $e_0$ as well, and is given by
    \begin{equation}
        e_0\pi^h_{bCas, \som \subset \slm}(\sigma_a) e_0=e_0e^{\pi {\rm i}(E_{a,a}+E_{a,a}E_{a+1,a+1})}e^{-\frac{\pi {\rm i}nh}{2}}\pi^{2h}_{KZ,\son}(\sigma_a)e_0.
    \end{equation}
    The element $F^{(m)}_{2h} e_0 (F^{(m)}_{2h})^{-1}$ corresponds to the projector onto the isotypic component of $V_{2h,\lambda^\top}^{\son}$ in the context of quantum orthogonal Howe duality from Theorem \ref{thm:quantumorthogonalHowe}. Therefore, its image admits the action of $U^{\iota}_{2h+\pi{\rm i}}(O_m)$ from Proposition \ref{b-action}.

    Now, by Theorem \ref{R-matrix-to-i-q-Weyl} and the Lemma \ref{iQW_pm_q2} we have
    \[
    \pi^{2\pi {\rm i} h}_{R,\son}(\sigma_a)=e^{\frac{\pi{\rm i}hn}{2}}\iQW_a(C_a)=e^{-\pi {\rm i}(E_{a,a}+E_{a,a}E_{a+1,a+1})}e^{\frac{\pi{\rm i}hn}{2}}\iQW_a(\widetilde{C}_a).
    \]
    Therefore, we get that
    \[
    F^{(m)}_{2h}e_0\pi^h_{bCas, \som \subset \slm}(\sigma_a) e_0 (F^{(m)}_{2h})^{-1} = F^{(m)}_{2h}e_0(F^{(m)}_{2h})^{-1} \iQW_a(\widetilde{C}_a) F^{(m)}_{2h}e_0 (F^{(m)}_{2h})^{-1}.
    \]

    So far we have only proved the equivalence of the monodromy representation $\pi^h_{bCas,\kompact\subset \g}$ and $ \pi_{bW, \kompact \subset \g}^{2\pi {\bf i}h}$ for modules $V^{O(m)}(\lambda)^{\oplus k}$ and $V^{O(m)}_{2h}(\lambda)^{\oplus k}$ for some $k$ over $O(m)$ and $U^{\iota}_{2h}(O(m))$ correspondingly. In order to obtain the statement of the Theorem for $k=1$, let us choose a projector $e_0'$ in the image of $U(\son)\rightarrow \End(\mathbb{S})$ corresponding to a one-dimensional image in $V^{\son}(\lambda^{\top})$. Since the conjugation of $e_0'$ by $F^{(m)}_{2h}$ yields a projector in the image of $U_{2h}(\son)\rightarrow \End(\mathbb{S})$, it commutes with the elements $\iQW_a(\widetilde{C}_a)$, so both representations are still isomorphic.
    
    We can further reduce it to the case of the actual irreducible modules over $U(\som)$ and $U_{2h}^{\iota}(\som)$ depending on whether $V^{O(m)}(\lambda)$ restricts to an irreducible $\som$-module, or $V^{O(m)}(\lambda)$ restricts to a direct sum of two irreducible $\som$-modules (see \cite[Proposition 10.3.]{MR4983792}), by recalling that at $h=0$: $F^{(m)}_{2h}={\rm Id}$ and $\widetilde{C}_a$ specialize to $B_{a,a+1}$. Consider the holomorphic family of algebras about $h=0$
    \[
    \End_{U^{\iota}_{2h}(\som)}(F^{(m)}_{2h}e_0'(F^{(m)}_{2h})^{-1}\mathbb{S}).
    \]
    It is commutative and $1$ or $2$-dimensional, therefore locally we can choose an idempotent $\widetilde{e}$ holomorphic in $h$ which projects onto a single irreducible summand $V^{O(m)}(\lambda)$. Since any $M$ from the statement of the Theorem is of this form, we are done.
\end{proof}

\appendix

\section{Technical proofs}

\subsection{Proof of Lemma \ref{lem:signforaap1}}\label{proof3}

    We will show that the action of permutation matrices $(a,a+1)\in  S_m\subset O_m$ on $\howespin$ are such that
    \[
    (a,a+1)\cdot \hwp_I=(-1)^{\#I_a\#I_{a+1}}\flip_{a,a+1}\cdot\hwp_I.
    \]

\begin{proof}
    According to \eqref{eq:Babactionsom}, the action of $\som$ on $\howespin\cong S^{\ot m}$ is given by even operators, therefore it is sufficient for us to treat the case when $m=2$. Let us take, for convenience of notation, $b:=-B_{12}=E_{21}-E_{12}\in \mathfrak{so}_2$ and compute the action of $e^{\frac{\pi}{2}b}$ on $\howespin\cong S^{\ot2}$. The action of $b$ on $\howespin$ is by
    \[
    \sum_{i=1}^n{}^{(i)}b, \quad {}^{(i)}b=\psi_{i2}^*\psi_{i1}-\psi_{i1}^*\psi_{i2}.
    \]
    Since the ${}^{(i)}b$ commute with each other, the action of $e^{\frac{\pi}{2}b}$ on $\howespin$ is by $\prod_{i=1}^n e^{\frac{\pi}{2}{}^{(i)}b}$, so we must compute the action of each $e^{\frac{\pi}{2}{}^{(i)}b}$ on $\hwp_I$, for $I\subset \{1, \dots, n\}\times \{1, 2\}$. 

Observe that
    \[
    {}^{(i)}b\cdot 1=0, \qquad {}^{(i)}b\cdot \psi^*_{i1}\psi^*_{i2}=0, \qquad {}^{(i)}b\cdot \psi_{i1}^*=\psi_{i2}^*, \qquad \text{and} \qquad {}^{(i)}b\cdot \psi_{i2}^*=-\psi_{i1}^*.
    \]
    Since
    \[
    \exp \left(\frac{\pi}{2}
    \begin{pmatrix}
        0 & -1\\
        1 & 0
    \end{pmatrix}
    \right)=
    \begin{pmatrix}
        0 & -1\\
        1 & 0
    \end{pmatrix},
    \]
we get (noting that $\psi^*_{i1}\psi^*_{i2}=-\psi^*_{i2}\psi^*_{i1}$)
\begin{equation}\label{eq:expbiaction1}
e^{\frac{\pi}{2}{}^{(i)}b}\cdot 1 = 1, \qquad e^{\frac{\pi}{2}{}^{(i)}b}\cdot \psi^*_{i1}\psi^*_{i2}=-\psi^*_{i2}\psi^*_{i1}, \qquad e^{\frac{\pi}{2}{}^{(i)}b}\cdot \psi_{i1}^* = \psi_{i2}^*, \qquad \text{and} \qquad e^{\frac{\pi}{2}{}^{(i)}b}\cdot \psi_{i2}^*=-\psi_{i1}^*.
\end{equation}
    
    Let us take a monomial $\hwp_I$, for $I\subset \{1, \dots, n\}\times \{1, 2\}$, and, for $1\le i\le n$, perform the following operation with it. Pull a monomial $m_1$ of the form $1,\psi^*_{i1},\psi^*_{i2}$ or $\psi^*_{i1}\psi^*_{i2}$ to the left, act on $\hwp_I=\pm m_1m_2$ by $e^{\frac{\pi}{2}{}^{(i)}b}$, which by \eqref{eq:expbiaction1} results in a monomial of the form $\pm1m_2,\pm \psi^*_{i2}m_2,\mp\psi^*_{i1}m_2$ or $\mp\psi^*_{i2}\psi^*_{i1}m_2$, respectively, and then move generators with the last index $2$ or $1$ to the original spot of the $1$ or $2$-generators correspondingly. Performing this operation for all $1\le i\le n$, results in the following
\[
e^{\frac{\pi}{2}b}\cdot \hwp_I=\prod_{i=1}^n e^{\frac{\pi}{2}{}^{(i)}b}\cdot \hwp_I = (-1)^{\#I_2}\prod^{\rightarrow}_{i_1\in I_1}\psi_{i_12}^* \prod^{\rightarrow}_{i_2\in I_2}\psi_{i_21}^*
\]
so
\[
e^{\frac{\pi}{2}b}\cdot \hwp_I= (-1)^{\#I_2}(-1)^{\#I_1\#I_2}\prod^{\rightarrow}_{i_2\in I_2}\psi_{i_21}^*\prod^{\rightarrow}_{i_1\in I_1}\psi_{i_12}^* = (-1)^{\#I_2}(-1)^{\#I_1\#I_2}\flip_{1,2}\cdot \hwp_I.
\]
Since in $O_2$, we have $(12)= e^{\pi {\bf i} E_{1,1}}\cdot e^{\frac{\pi}{2}b}$ and 
\[
e^{\pi {\bf i} E_{1,1}}\cdot \flip_{1,2}\cdot \hwp_I = (-1)^{\#I_2}\flip_{1,2}\cdot\hwp_I,
\]
it follows that
\[
(12)\cdot \hwp_I=e^{\pi {\bf i} E_{1,1}}\cdot e^{\frac{\pi}{2}b}\cdot \hwp_I= (-1)^{\#I_1\#I_2}\flip_{1,2} \cdot\hwp_I.
\]
\end{proof}

\subsection{Proof of Proposition \ref{C-action}}\label{proof5.12}

Recall, we aim to show that 
    \[
        C\hwp^{\pm \pm}_k= [2k]_Q \hwp_k^{\mp\mp} \quad \text{and} \quad C\hwp^{\pm \mp}_k= [2k+1]_Q \hwp_k^{\mp\pm}.
    \]

\begin{remark}
Since $Q=-q^2$, we have $[2k]_Q= -\frac{[4k]_q}{[2]_q}$ and $[2k+1]_Q= +\frac{[4k+2]_q}{[2]_q}$.    
\end{remark}

\begin{proof}[Proof of Proposition \ref{C-action}]
Clearly, $C\hwp_0^{\pm \pm}=0$, $C\hwp_0^{\pm\mp}=\hwp_0^{\mp\pm}$, and more generally,
$C\hwp_{k}^{\pm \pm}$ is proportional to $\hwp_{k}^{\mp \mp}$ and $C\hwp_{k}^{\pm \mp}$ is proportional to $\hwp_{k}^{\mp \pm}$. To precisely determine the proportions is a direct computation, e.g.,
\[
    C\hwp^{++}_k = \hwp^{--}_k ((-q^2)^{-(2k-1)}+\sum_{i=n-2k+1}^{n-1}(-q^2)^{n-i}(-q^2)^{n+1-2k-i})=[2k]_Q\hwp^{--}_k.
\]
The other cases are similar.
\end{proof}

\subsection{Proof of Proposition \ref{R-action}}\label{proof6.1}

Recall that we aim to show 
\[
    R_{S,S}(\hwp_k^{\pm \pm})=(-1)^kq^{\frac{n-2(2k)^2}{2}} \hwp_k^{\pm \pm}, \quad R_{S,S}(\hwp_k^{\pm \mp})=(-1)^kq^{\frac{n-2(2k+1)^2}{2}} \hwp_k^{\mp \pm}.
\]

\begin{proof}[Proof of Proposition \ref{R-action}]
 We first consider two perspectives on the braiding
 \[
 R=\flip\circ q^{(\wt(-),\wt(-))}\circ (T_{w_0}^{-1} \ot T_{w_0}^{-1} 
 )\circ \Delta(T_{w_0}),
 \]
 as in \eqref{R-matrix}. Let $V,W$ be type I $U_h(\g)$-modules. On the one hand, $R$ acts as a $U_h(\g)$-module intertwiner $V\otimes W\rightarrow W\otimes V$. The key takeaway is that $R$ sends highest weight vectors to highest weight vectors. On the other hand, if $v\in V[\lambda]$ and $w\in W[\mu]$, then
 \[
(T_{w_0}^{-1} \ot T_{w_0}^{-1} 
 )\circ \Delta(T_{w_0})\cdot v\otimes w= v\otimes w + V[<\lambda]\otimes W[> \mu],
 \]
as in \cite[Proposition 3.5.]{BER} (see also Remark \ref{Hopf_quasitriangular}), and therefore
 \begin{equation}\label{eq:Risscalarflip-unitri}
 R_{V,W}\cdot v\otimes w = q^{(\lambda, \mu)}w\otimes v + W[>\mu]\otimes V[<\lambda].
 \end{equation}
 The key takeaway is that we can use weight considerations to compute with $R$. 
 
 We will now argue separately for each case of $\hwp_k^{\pm\pm}$ and $\hwp_k^{\pm\mp}$.

\smallskip
\noindent\textbf{Part (a): The case of $\hwp_k^{++}$.}\\

We know that $R_{S^+,S^+}$ is an intertwiner, so sends highest weight vectors to highest weight vectors. Thus,
\begin{equation}\label{eq:R++scalar}
    R_{S^+,S^+}(\hwp_k^{++})=\beta\hwp_k^{++}
\end{equation}
and we wish to find the scalar $\beta$. For this, we first observe that in \eqref{eq:highvect++}, we have
\[
\hwp_k^{++}=c^+_{\emptyset}\hwp_{\emptyset}\otimes \hwp_{I_{2k}} + (\dots),
\]
where $(\dots)$ is a sum of terms in $S^{\otimes 2}$ such that all have the second tensor factor in weight strictly higher (in dominance order) than the weight of $\hwp_{I_{2k}}$. Compare the outcome of the action by $R_{S^+,S^+}$ to find
\[
q^{(\lambda_{\emptyset},\lambda_{I_{2k}})}c^+_{\emptyset}\hwp_{I_{2k}} \otimes \hwp_{\emptyset}=\beta c^+_{I_{2k}}\hwp_{I_{2k}} \otimes \hwp_{\emptyset},
\]
(here the left hand uses \eqref{eq:Risscalarflip-unitri} and the right hand side uses \eqref{eq:R++scalar}). We conclude that
\[
    \beta=q^{\frac{n-4k}{2}}(-q^2)^{-2k^2+k}=(-1)^kq^{\frac{n-8k^2}{2}}.
\]

\smallskip
\noindent\textbf{Part (b): The case of $\hwp_k^{--}$.}\\

If $k=0$, then
\[
    \beta = q^{(\lambda_{\{n\}},\lambda_{\{n\}})}=q^{\frac{n}{2}}.
\]
Otherwise, 
\[
    q^{(\lambda_{\{n\}},\lambda_{I_{2k}\setminus \{n\}})}c^-_{\{n\}}\hwp_{I_{2k}\setminus \{n\}}\otimes \hwp_{\{n\}} =\beta c^-_{I_{2k}\setminus \{n\}}\hwp_{I_{2k}\setminus \{n\}}\otimes \hwp_{\{n\}},
\]
so
\[
    \beta=q^{\frac{n-4k}{2}}Q^{-2k^2+k}=(-1)^kq^{\frac{n-8k^2}{2}}.
\]

\smallskip
\noindent\textbf{Part (c): The case of $\hwp_k^{+-}$.}\\

Here we have $R_{S^+,S^-}\hwp_{k}^{+-}=\beta \hwp_{k}^{-+}$. Therefore,
\[
    q^{(\lambda_{\emptyset},\lambda_{I_{2k+1}})}Q^{\hgt(\lambda_{low,+}-\lambda_{\emptyset})}\hwp_{I_{2k+1}}\otimes \hwp_{\emptyset}=\beta Q^{\hgt(\lambda_{low,-}-\lambda_{I_{2k+1}})}\hwp_{I_{2k+1}}\otimes \hwp_{\emptyset},
\]
so
\[
    \beta=q^{(\lambda_{\emptyset},\lambda_{I_{2k+1}})}Q^{\hgt(\lambda_{low,+}-\lambda_{low,-}+\lambda_{I_{2k+1}}-\lambda_{\emptyset})}=q^{\frac{n-4k-2}{2}}(-q^2)^{-2k^2-k}=(-1)^kq^{\frac{n-2(2k+1)^2}{2}}.
\]

\smallskip
\noindent\textbf{Part (d): The case of $\hwp_k^{-+}$.}\\

Let $R_{S^-,S^+}\hwp_{k}^{-+}=\beta \hwp_{k}^{+-}$. Therefore,
\[
    q^{(\lambda_{\{n\}},\lambda_{I_{2k+1}\setminus \{n\}})}Q^{\hgt(\lambda_{low,-}-\lambda_{\{n\}})}\hwp_{I_{2k+1}\setminus \{n\}}\otimes \hwp_{\{n\}}=\beta Q^{\hgt(\lambda_{low,+}-\lambda_{I_{2k+1}\setminus \{n\}})}\hwp_{I_{2k+1}\setminus \{n\}}\otimes \hwp_{\{n\}},
\]
so
\[
\beta=(-1)^kq^{\frac{n-2(2k+1)^2}{2}}.
\]
\end{proof}

\subsection{Proof of Proposition \ref{i-q_gens}}\label{proof6.2}

Recall, we aim to show that the following identities hold
\[
    \iQW^{-1}_{\overline{0}}([2k]_Q)=(-1)^k Q^{2k^2} \qquad \text{and} \qquad \iQW^{-1}_{\overline{1}}([2k+1]_Q)=(-Q)^{1/2} (-1)^kQ^{2(k^2+k)}.
\]

We first give the following lemma on quantum number arithmetic.

\begin{lemma}\label{lem:quantumarith}
    For $i,j\in \mathbb{Z}$, we have
\begin{equation}\label{eq:quantaritheven}
    [2i]_Q^2-[2j]_Q^2=[2i+2j]_Q[2i-2j]_Q
\end{equation}
and
\begin{equation}\label{eq:quantuarithodd}
    [2i+1]_Q^2-[2j+1]^2_Q=[2i+2j+2]_Q[2i-2j]_Q.
\end{equation}
\end{lemma}
\begin{proof}
    For \eqref{eq:quantaritheven}, note that
    \begin{align*}
    [2i]_Q^2-[2j]_Q^2&=\frac{Q^{4i}-2 +Q^{-4i}-Q^{4j}+2 -Q^{-4j}}{(Q-Q^{-1})^2} \\
    &=\frac{(Q^{2i+2j}-Q^{-2i-2j})(Q^{2i-2j}-Q^{2j-2i})}{(Q-Q^{-1})^2}\\
    &=[2i+2j]_Q[2i-2j]_Q.
    \end{align*}
    The argument for \eqref{eq:quantuarithodd} is similar.
\end{proof}

Now we can write the $\iota$divided power polynomials, which are products of difference of squares of quantum numbers, as honest products of quantum numbers.

\begin{lemma}
For $i\in \mathbb{Z}_{\ge 0}$ and $k\in \mathbb{Z}$, we have
\begin{equation}
    B_{\overline{0}}^{(2i)}([2k]_Q)=\frac{[2k]_Q^2[2k-2]_Q[2k+2]_Q\dots [2k-2i+2]_Q[2k+2i-2]_Q}{[2i]_Q!},
\end{equation}
and
\begin{equation}
    B_{\overline{1}}^{(2i+1)}([2k+1]_Q)=[2k+1]_Q\frac{[2k]_Q[2k+2]_Q\dots [2k-2i+2]_Q[2k+2i]_Q}{[2i+1]_Q!}.
\end{equation}
\end{lemma}
\begin{proof}
    Apply Lemma \ref{lem:quantumarith} to \eqref{eq:evendivpowerpoly} and \eqref{eq:odddivpowerpoly}. 
\end{proof}

Next, we derive a recursion for the $\iota$divided power polynomials.

\begin{lemma}\label{lem:divpowerpolyrecursion}
For $i\in \mathbb{Z}_{\ge 0}$ and $k\in \mathbb{Z}$, we have
    \begin{equation}\label{eq:evendivpolyrecursion}
    B_{\overline{0}}^{(2i)}([2k+2]_Q)-B_{\overline{0}}^{(2i)}([2k]_Q)=(Q^{2k+1}+Q^{-2k-1})B_{\overline{1}}^{(2i-1)}([2k+1]_Q)
\end{equation}
and
\begin{equation}\label{eq:odddivpolyrecursion}
    B_{\overline{1}}^{(2i+1)}([2k+1]_Q)-B_{\overline{1}}^{(2i+1)}([2k-1]_Q)=(Q^{2k}+Q^{-2k})B_{\overline{0}}^{(2i)}([2k]_Q).
\end{equation}
\end{lemma}
\begin{proof}
To establish \eqref{eq:evendivpolyrecursion}, observe that
\begin{gather*}
    B_{\overline{0}}^{(2i)}([2k+2]_Q)-B_{\overline{0}}^{(2i)}([2k]_Q)\\
    =\\
    \frac{[2k-2i+4]_Q\dots [2k+2i-2]_Q}{[2i]_Q!}([2k+2]_Q[2k+2i]_Q-[2k]_Q[2k-2i+2]_Q),
\end{gather*}
and 
\begin{gather*}
    [2k+2]_Q[2k+2i]_Q-[2k]_Q[2k-2i+2]_Q\\
    =
    \\
    \frac{Q^{4k+2i+2}-Q^{2-2i}-Q^{2i-2}+Q^{-4k-2i-2}-Q^{4k-2i+2}+Q^{-2i+2}+Q^{2i-2}-Q^{-4k+2i-2}}{(Q-Q^{-1})^2}\\
    =\\
    [4k+2]_Q[2i]_Q=(Q^{2k+1}+Q^{-2k-1})[2k+1]_Q[2i]_Q.
\end{gather*}
Similarly, for \eqref{eq:odddivpolyrecursion}, use
\begin{gather*}
    [2k+1]_Q[2k+2i]_Q-[2k-1]_Q[2k-2i]_Q\\
    =\\
    \frac{Q^{4k+2i+1}-Q^{1-2i}-Q^{2i-1}+Q^{-4k-2i-1}-Q^{4k-2i-1}+Q^{1-2i}+Q^{2i-1}-Q^{-4k+2i+1}}{(Q-Q^{-1})^2}\\
    =\\
    [4k]_Q[2i+1]_Q\\
    =\\
    (Q^{2k}+Q^{-2k})[2k]_Q[2i+1]_Q.
\end{gather*}
\end{proof}

Now, we will use the recursion from Lemma \ref{lem:divpowerpolyrecursion} to prove the Proposition.

\begin{proof}[Proof of Proposition \ref{i-q_gens}]
For convenience of notation, let $\widetilde{\iQW^{-1}_{\overline{1}}}:=(-Q)^{-1/2}\iQW^{-1}_{\overline{1}}$, so
\[
\iQW^{-1}_{\overline{0}} = 1- QB_{\overline{0}}^{(2)} + \cdots \quad \text{and} \quad \widetilde{\iQW^{-1}_{\overline{1}}} = B_{\overline{1}} - QB_{\overline{1}}^{(3)} + \cdots .
\]

We will prove the claim for $k\in \mathbb{Z}_{\ge 0}$ by induction. The base for the induction is
\[
    \iQW^{-1}_{\overline{0}}([0]_Q)=1 \quad \text{and} \quad \widetilde{\iQW^{-1}_{\overline{1}}}([1]_Q)=1.
\]
Suppose for induction that
\begin{equation}\label{eq:inductionappendix}
    \iQW^{-1}_{\overline{0}}([2k]_Q)=(-1)^k Q^{2k^2}\quad \text{and} \quad \widetilde{\iQW^{-1}_{\overline{1}}}([2k-1]_Q)=(-1)^{k-1}Q^{2(k-1)k}.
\end{equation}
First, we compute
\begin{align*}
    \widetilde{\iQW^{-1}_{\overline{1}}}([2k+1]_Q) &\stackrel{\eqref{eq:odddivpolyrecursion}}{=}\widetilde{\iQW^{-1}_{\overline{1}}}([2k-1]_Q) + (Q^{2k}+Q^{-2k})\iQW^{-1}_{\overline{0}}([2k]_Q) \\
    &\stackrel{\eqref{eq:inductionappendix}}{=}(-1)^{k-1}Q^{2(k^2-k)} +(Q^{2k}+Q^{-2k})(-1)^kQ^{2k^2}\\
    &=(-1)^kQ^{2k^2+2k} \\
    &=(-1)^kQ^{2k(k+1)}
\end{align*}
Then, using the previous computation, we compute 
\begin{align*}
    \iQW^{-1}_{\overline{0}}([2k+2]_Q)&\stackrel{\eqref{eq:evendivpolyrecursion}}{=}\iQW^{-1}_{\overline{0}}([2k]_Q)-Q(Q^{2k+1}+Q^{-2k-1}) \widetilde{\iQW^{-1}_{\overline{1}}}([2k+1]_Q)\\
    &=\iQW^{-1}_{\overline{0}}([2k]_Q)-Q(Q^{2k+1}+Q^{-2k-1})(-1)^k Q^{2(k^2+k)}\\
    &\stackrel{\eqref{eq:inductionappendix}}{=}(-1)^kQ^{2k^2}-Q(Q^{2k+1}+Q^{-2k-1})(-1)^k Q^{2(k^2+k)}\\
    &=(-1)^{k+1}Q^{2k^2+4k+2}\\
    &=(-1)^{k+1}Q^{2(k+1)^2}.
\end{align*}
This finishes the proof for $k\in \mathbb{Z}_{\ge 0}$. We leave it to the reader to deduce the result for all $k\in \mathbb{Z}$.
\end{proof}

\subsection{Proof of Lemma \ref{iQW_pm_q2}}\label{prooflem6.4}

We aim to show that
 \[
        e^{\pi {\bf i}E_{i,i}}(i,i+1)\flip^{-1}_{i,i+1} \iQW_i(C_i)=\iQW_i(\widetilde{C}_i). 
    \]

\begin{proof}[Proof of Lemma \ref{prooflem6.4}]

It suffices to prove the claim when $m=2$ and $i=1$. By Lemma \ref{lem:signforaap1}, we have $(1,2)\flip_{1,2}^{-1}=(-1)^{\#I_1\#I_2}$ which is $+1$ on $S^-\otimes S^+$, $S^+\otimes S^-$, and $S^+\otimes S^+$, and is $-1$ on $S^-\otimes S^-$. Therefore,
\[
l_1 :=e^{\pi {\bf i}E_{1,1}}(1,2)\flip^{-1}_{1,2}=\begin{cases}
    e^{\pi {\bf i}E_{1,1}}, & \text{ on } S^{\otimes 2}_{\overline{1}}:=(S^-\ot S^+) \oplus (S^+ \ot S^-),\\
    1, & \text{ on } S^{\otimes 2}_{\overline{0}}:=(S^+ \ot S^+) \oplus (S^- \ot S^-).
\end{cases}
\]
We will re-express $l_1 \iQW_1$, originally written in terms of $C$ and $Q=-q^2$, by using $\widetilde{C}$ and $q^2$. 

Recall that
\[
\widetilde{C}={\bf i}e^{\pi {\bf i}E_{1,1}} C,
\]
and note that
\[
[2k]_Q=-[2k]_{q^2} \qquad \text{and} \qquad [2k+1]_Q=+[2k+1]_{q^2},
\]
for $k\in \mathbb{Z}$.

For the even part, we observe that
\[
\widetilde{C}^2=-e^{\pi {\bf i}E_{1,1}} C e^{\pi {\bf i}E_{1,1}} C= C^2, 
\]
so
\begin{align*}
\iQW^{-1}_{\bar{0},Q}(C)&\stackrel{\eqref{eq:evenseries}}{=}1-Q B_{\bar{0},Q}^{(2)}(C)+Q^2B_{\bar{0},Q}^{(4)}(C)-\dots \\
&= 1-q^2 B_{\bar{0},q^2}^{(2)}(\widetilde{C})+(q^2)^2B_{\bar{0},q^2}^{(4)}(\widetilde{C})-\dots\\
&= \iQW^{-1}_{\bar{0},q^2}(\widetilde{C}).
\end{align*}
Let us choose a positive branch of the root in the equation \eqref{iQ_symm}. For the odd part, we have
\begin{align*}
    e^{\pi {\bf i}E_{1,1}}\iQW^{-1}_{\bar{1},Q}(C)&\stackrel{\eqref{eq:oddseries}}{=} e^{\pi {\bf i}E_{1,1}} (-Q)^{1/2}\big(B_{\bar{1},Q}^{(1)}(C)-QB_{\bar{1},Q}^{(3)}(C)+Q^2B_{\bar{1},Q}^{(5)}(C)-\dots\big)\\
    &=\frac{1}{{\bf i}} (q^2)^{1/2}\big(B_{\bar{1},q^2}^{(1)}(\widetilde{C})-q^2B_{\bar{1},q^2}^{(3)}(\widetilde{C})+(q^2)^2B_{\bar{1},q^2}^{(5)}(\widetilde{C})-\dots\big)\\
    &=-(-q^2)^{1/2}\big(B_{\bar{1},q^2}^{(1)}(\widetilde{C})-q^2B_{\bar{1},q^2}^{(3)}(\widetilde{C})+(q^2)^2B_{\bar{1},q^2}^{(5)}(\widetilde{C})-\dots\big) \\
    &= -\iQW^{-1}_{\bar{1},q^2}(\widetilde{C}).
\end{align*}
Therefore,
\[
e^{\pi {\bf i}E_{1,1}}\iQW^{-1}_{\bar{1},Q}(C)=e^{-\pi {\bf i}E_{1,1}}\iQW^{-1}_{\bar{1},Q}(C)=-\iQW^{-1}_{\bar{1},Q}(C)e^{-\pi {\bf i}E_{1,1}}=\iQW^{-1}_{\bar{1},q^2}(\widetilde{C}) 
\]
So $l_1 \iQW(C)=\iQW(\widetilde{C})$ on $S^{\ot 2}$. 

\end{proof}

\bibliographystyle{plain}
\bibliography{2025arxiv}

\end{document}